\documentclass[letterpaper, dvipsnames, 12pt]{article}
\usepackage[margin=1in]{geometry}
\usepackage[utf8]{inputenc}
\usepackage[T1]{fontenc}
\usepackage{amsmath,amsthm,amssymb}

\usepackage{graphicx}
\usepackage{tikz}
\usepackage{hyperref}
\usepackage{color}
\usepackage{enumerate}
\usepackage{multicol}
\usepackage{xcolor}
\usepackage{amsrefs}
\newcommand{\SL}{\sum\limits_}
\newcommand{\PL}{\prod\limits_}

\newcommand{\R}{\mathbb{R}}
\newcommand{\N}{\mathbb{N}}
\newcommand{\Z}{\mathbb{Z}}
\newcommand{\PR}{\mathbb{P}}
\newcommand{\E}{\mathbb{E}}

\newcommand{\sgn}{\text{sgn}}

\usepackage[autostyle,english=american]{csquotes}
\MakeOuterQuote{"}

\newtheorem{ccounter}{ccounter}[section]
\newtheorem{thm}[ccounter]{Theorem}
\newtheorem{lem}[ccounter]{Lemma}
\newtheorem{cor}[ccounter]{Corollary}
\newtheorem{defn}[ccounter]{Definition}
\newtheorem{prop}[ccounter]{Proposition}

\newtheorem{rem}[ccounter]{Remark}

\title{Distribution of the magnetization of the critical Ising model on sparse random graphs}
\author{Kyprianos-Iason Prodromidis\thanks{Department of Mathematics, Princeton University, Email: kp2702@princeton.edu} \and Allan Sly\thanks{Department of Mathematics, Princeton University, Email: allansly@princeton.edu}}
\date{}
\begin{document}
\maketitle
\begin{abstract}
In this paper, we consider the Ising model on random $d$-regular graphs (with $d\ge3$) and Erd\"os-R\'enyi graphs $G(n,d/n)$ (with $d>1$) at the critical temperature. We prove that the \textit{magnetization}, i.e.\ the sum of the spins of a configuration, is typically of order $n^{3/4}$ and when multiplied by $n^{-3/4}$ converges in distribution to a non-trivial random variable, whose density we describe. In the regular graph case, the Small Subgraph Conditioning Method applies, and the limiting density is of the form $\frac1{Z}\exp(-C_d z^4)$.

Surprisingly, in the Erd\"os-R\'enyi case, while the ratio of the second moment and first moment squared is bounded, the short cycle count is not enough to explain the fluctuations of the partition function restricted to a particular magnetization. We identify the additional source of randomness as path counts of slowly diverging length. This quantity is motivated by the heuristic that correlations between distant vertices are proportional to their local branching rate. Augmenting the Small Subgraph Conditioning Method with these path counts allows us to prove convergence of the magnetization to a non-deterministic limiting distribution. To our knowledge, the need to condition on graph observables beyond the cycle counts is a new phenomenon for spin systems.

As further corollaries, we derive a polynomial lower bound on the mixing time of the stochastic Ising model on sparse random graphs at the critical temperature complementing recent upper bounds~\cite{Polchinski-eq,P-S}.  Moreover, we establish the fluctuations of the free energy in the Erd\"os-R\'enyi case, answering a recent question of Coja-Oghlan et. al.~\cite{Coja-et-al}.
\end{abstract}

\section{Introduction}

Just as for the lattice $\Z^d$, on sparse random graphs the Ising model exhibits a phase transition with exponential decay of correlation at high temperatures and long-range order emerging at low temperatures. Because these graphs are locally treelike, many properties of the Ising model can be understood in terms of the Ising model on its local weak limit graph.
In the low temperature regime, the local weak limit of the Ising model and graph together was shown to be a balanced mixture of the extremal plus and minus measures on the local weak limit tree (either the infinite $d$-regular tree for random $d$-regular graphs or the Galton-Watson branching process with Poisson offspring distribution on $G(n,d/n)$) and the free energy is explicitly known \cite{BasakDem,D-M,GerMon,MMS}.

We are interested in the law of the magnetization.  The simplest graph to study the Ising model is the \textit{mean-field} setting, where the underlying graph is the complete graph. In this case, if $n$ is the number of vertices and the inverse temperature is $\beta/n$, the model undergoes a similar phase transition at $\beta=1$ (see for example \cite{Ellis2006}). When $\beta<1$,  the magnetization $m_n=\SL{i=1}^n\sigma_i$ obeys a central limit theorem with
\[
m_n\cdot n^{-1/2}\xrightarrow[n\to\infty]{(d)}N(0,(1-\beta)^{-1}).
\]
On the other hand, when $\beta>1$, the distribution is bimodal: $\frac1n m_n$ concentrates near $\pm x$, where $x$ solves the equation $x=\tanh(\beta x)$. Finally, in the critical case $\beta=1$, $m_n$ is of order $n^{3/4}$ and $m_n\cdot n^{-3/4}$ converges in distribution to a law with density $\varpropto\exp(-y^4/12)$ (see \cite{Simon-Griffiths}). In the critical window when $\beta = 1+\theta n^{-1/2}$, for some $\theta\in\R$, the limiting density becomes $\varpropto\exp(\theta y^2/2-y^4/12)$.

The case of the distribution of the magnetization for the critical Ising model in $\Z^d$ remains unresolved. In dimensions 5 and higher it is natural to expect behaviour similar to the mean field case. There is, however, a subtlety in the effect of the boundary conditions. In a box with free boundary conditions the magnetization is in fact Gaussian. This is because the missing edges on the boundary slightly lower the average degree and effectively push the system into the high temperature regime.  This is not an issue with periodic boundary conditions in which case~\cite{LPS} showed that there is a non-normal magnetization and conjectured that it has asymptotic density $\exp(-cy^4)$ as in the complete graph. The case of $d=2$ is quite different, as the limiting distribution was established in~\cite{CGN} and has tails $\exp(-(c+o(1))x^{16})$.

In this paper, we study the same problem at criticality on sparse random graphs, more specifically on random $d$-regular graphs (for $d\ge3$) and Erd\"os-R\'enyi graphs $G(n,d/n)$ (for $d>1$).  In both cases, at the critical temperature we find that the limiting law of the magnetization is non-Gaussian.

To describe convergence to a random measure, we will consider convergence in the Wasserstein distance with respect to a Kolmogorov–Smirnov metric. Specifically, on the space $\mathcal{M}_1(\R)$ of probability measures on $\R$, consider the metric
$$\text{d}_{\text{KS}}(\mu,\nu):=\sup\limits_{y\in\R}|\mu((-\infty,y])-\nu((-\infty,y])|.$$
We are interested in the 1-Wasserstein distance associated to the $\text{d}_{\text{KS}}$ metric.

In the random regular case the limit is, in fact, deterministic, and the same as for the critical Curie-Weiss model, up to a $d$-dependent constant in the exponent.

\begin{thm}\label{main-thm-1}
Consider the Ising model on a graph $G\sim G_{n,d}$ (for $d\ge3$) at the critical temperature, $\beta_c=\tanh^{-1}((d-1)^{-1})$. Let $\mu_{n,\text{reg}}\in\mathcal{M}_1(\R)$ be the random measure defined as
$$\mu_{n,\text{reg}}(A):=\PR_{n,\text{reg}}(m\cdot n^{-3/4}\in A\ |\ G),\ \text{with}\ G\sim G_{n,d}.$$
Then, there exists a measure $\mu_{\text{reg}}\in\mathcal{M}_1(\R)$ such that $\mu_{n,\text{reg}}\xrightarrow[n\to\infty]{}\mu_{\text{reg}}$ in the 1-Wasserstein distance associated to $\text{d}_{\text{KS}}$. This measure $\mu_{\text{reg}}$ has density $$f_{\text{reg}}(y)\varpropto \exp\left(-\frac{(d-2)(d-1)}{12d^2}y^4\right).$$
\end{thm}

The case of the Erd\'os-R\"enyi random graph is more complicated; the limiting distribution is not deterministic as in the random regular case. The limit takes the form of a mixture of the distributions found in the critical window for the Curie-Weiss model.

\begin{thm}\label{main-thm-2}
Consider the Ising model on a graph $G\sim G(n,d/n)$ (for $d>1$) at the critical temperature, $\beta_c=\tanh^{-1}(d^{-1})$. Let $\mu_n\in\mathcal{M}_1(\R)$ be the random measure defined as 
$$\mu_n(A):=\PR(m\cdot n^{-3/4}\in A\ |\ G), \ \text{with}\ G\sim G(n,d/n).$$
Then, there exists a random measure $\mu\in\mathcal{M}_1(\R)$ such that $\mu_n\xrightarrow[n\to\infty]{}\mu$ in the Wasserstein distance associated to $\text{d}_{\text{KS}}$. Moreover, the measure $\mu$ is given by $\mu=\mu^{(X)}$, where for $x\in\R$,
$$\mu^{(x)}(A)=\dfrac{\int_A\exp\left(\frac{x}{\sqrt{2(d-1)}}\cdot u^2-\left(\frac{1}{4(d-1)}+\frac{1}{12}\right)\cdot u^4\right) \ \text{d}u}{\int_\R\exp\left(\frac{x}{\sqrt{2(d-1)}}\cdot u^2-\left(\frac{1}{4(d-1)}+\frac{1}{12}\right)\cdot u^4\right) \ \text{d}u},$$
and $X\sim N(0,1)$.
\end{thm}

The natural approach for studying the magnetization of spin systems on random graphs is to apply the second moment method to the partition function and its restriction to a particular magnetization. One commonly finds that
\[
\dfrac{\E(Z_n^2)}{\E(Z_n)^2}\xrightarrow[n\to\infty]{}c\in(1,\infty)
\]
implying that the normalized partition function retains some non-trivial asymptotic variance.  This variance is typically explained by the effect of small cycles in the graph as was first demonstrated by Robinson and Wormald in the study of Hamiltonicity of random regular graphs~\cite{RobinsonWormald1994} where they introduced the Small Subgraph Conditioning Method.
This method has been applied to a range of spin systems such as the ferromagnetic (e.g. \cite{Coja-et-al}) and anti-ferromagnetic (e.g. \cite{Fabian-Loick}) Ising model, the hardcore model (e.g. \cite{MosselWeitzWormald2009}) and colorings (e.g. \cite{DiazEtAl2009,KemkesPerezWormald2010}).
The heuristic explanation for the small graph conditioning method is that adding an edge of a graph that closes a cycle has a different effect on the partition function than adding an edge between distant vertices; in the uniqueness regime, nearby vertices are correlated while distant ones are almost uncorrelated. For each $i$, the number of cycles of length $i$ in the graph, $Y_{i,n}$, is asymptotically Poisson with mean $\lambda_i$ with a jointly independent limit.  For appropriately chosen $\alpha_i, \kappa_i$ and a slowly growing $m_n$, one can apply the second moment method to
\[
Z_n \prod_{i=3}^{m_n}\exp(-\alpha_i C_i+\kappa_i)
\]
and show that it is asymptotically constant. Consequently, for $C_i$ independent Poisson$(\lambda_i)$,
\[
\dfrac{Z_n}{\E(Z_n)}\xrightarrow[n\to\infty]{(d)}W=\exp\left(\SL{i=1}^\infty(\alpha_iC_i-\kappa_i)\right)
\]
The same method can be applied to the partition function, restricted to a given magnetization, which in turn yields the asymptotic law of the magnetization.

This method has been previously applied for the high-temperature and the anti-ferromagnetic regime up to the reconstruction threshold. It is, therefore, reasonable to assume that it should also apply at the critical temperature where decay of correlation still holds and indeed it does for the regular graph case.  In the case of the Erd\"os-R\'enyi random graph, it is necessary to account for the fluctuations in the total number of edges. This can be done in two ways: one is to fix the number of edges (which is the route taken in \cite{Coja-et-al}) and the second is to reweight the partition function by a factor of $(\cosh(\beta))^{-|E|}$, as each extra edge increases the partition function by approximately a $\cosh(\beta)$ factor.  After either of these adjustments, in the high temperature regime, the Small Subgraph Conditioning Method works as normal. However, at the critical temperature, even after adjusting for the edge count and the effect of small cycles, a non-trivial amount of variance remains asymptotically. In fact, when $d\in(1,4)$,
\[
\dfrac{\E[(Z_n(\cosh(\beta))^{-|E_n|})^2]}{\E[(Z_n(\cosh(\beta))^{-|E_n|}]^2}\to\infty.
\]
However, if we only consider configurations whose magnetizations are $m$, i.e. we set
$$Z_{n,m}=\SL{\sigma:\ \SL{v}\sigma_v=m}\exp\left(\beta\SL{u\sim v}\sigma_u\sigma_v\right)$$
and assume that $m\cdot n^{-3/4}\to x$ then
\[
\dfrac{\E[(Z_{n,m}(\cosh(\beta))^{-|E_n|})^2]}{\E[(Z_{n,m}(\cosh(\beta))^{-|E_n|}]^2}\to c(x,d)\in(1,\infty).
\]
To understand the source of this surprising extra variance, consider the process of building the graph one edge at a time with $G_i$ the graph after $i$ edges.  When adding the $i+1$ edge $(u,v)$, the change in the partition function is given by
\[
\frac{Z_{G_{i+1}}}{Z_{G_i}}=\cosh(\beta)\left(1+\tanh(\beta)\cdot\E_{G_i}(\sigma_u\sigma_v)\right)
\]
Note that if the edge closes a small cycle, it means that $u$ and $v$ were close in $G_i$ and therefore $\E_{G_i}(\sigma_u\sigma_v)\approx (\tanh(\beta))^{d(u,v)}$. This is the cycle effect in the Small Subgraph Conditioning Method. Two random vertices will typically have $\E_{G_i}(\sigma_u\sigma_v)\asymp n^{-\frac12}$ at the critical temperature (at least for the final edges of the construction). Moreover, how correlated $\sigma_u$ and $\sigma_v$ are will depend on how well $u$ and $v$ are connected to the giant component of the graph.  Heuristically, from a consideration of the critical FK-component, we expect  
\begin{equation}\label{eq:covariance.heuristic}
\E_{G_i}(\sigma_u\sigma_v)\asymp n^{-\frac12}d^{-2\ell}S_\ell(u) S_\ell(v),
\end{equation}
where $S_\ell(u)$ is the number of vertices at distance $\ell$ from $u$.  This suggests that the variance of $\E_{G_i}(\sigma_u\sigma_v)$ is of order $C/n$ and so summed over order $n$ edges these weak correlations should induce order one multiplicative fluctuations in the partition function which explains the failure of the vanilla implementation of the Small Subgraph Conditioning Method.

While we won't need to prove~\eqref{eq:covariance.heuristic}, it suggests the right adjustment.  If the covariance is proportional to $\E_{G_i}(\sigma_u\sigma_v)$ then it is also approximately proportional to the number of additional paths of length $2\ell+1$ created when adding the edge $(u,v)$. This approximation is more accurate for larger values of $\ell$.  We let $X_{\ell,n}$ denote the number of paths of length $\ell$ present in $G$ and 
\begin{equation}\label{hat-X-def}
\widehat{X}_{\ell,n}:=\dfrac{X_{\ell,n}-\frac{1}{2}nd^\ell}{\sqrt{\frac12nd^{2\ell}\frac{\ell^2}{d-1}}}.
\end{equation}
While we don't prove it, as $n\to \infty$, the $(\widehat{X}_{\ell,n})_{\ell\geq 3}$ converges jointly in distribution to a jointly Gaussian ensemble of normal random variables $(X^*_\ell)_{\ell\geq 3}$.  Moreover, that $X^*_\ell$ converge in probability as $\ell\to\infty$ to some standard normal random variable $X^*_\infty$.  This quantity gives a measure of the amount of long paths in the graph and explains the additional source of the variance.  As such, our approach is to apply the Small Subgraph Conditioning Method with $X_{\ell_n,n}$ in addition to the cycle counts for some slowly diverging $\ell_n$.  As far as we know, this is the first example of a spin system where augmenting the cycles count with an additional observable of the graph is necessary to apply the Small Subgraph Conditioning Method (see \cite{RobinsonWormald2001} in the case of Hamiltonian cycles with randomly selected oriented edges requires extra observables from the edge decorations).

We note that Theorems \ref{main-thm-1} and \ref{main-thm-2} also imply lower bounds for the mixing times of the critical stochastic Ising model for both regular and Erd\"os-R\'enyi graphs. This is the first polynomial lower bound for this chain, and complements upper bounds proven recently \cite{Polchinski-eq, P-S}.  The proof uses the magnetization as a test function giving a lower bound on the spectral gap, using the fact that lower-bounding the variance of the magnetization is enough to show a lower bound for the relaxation time (see e.g.~\cite{Holley,Lubetzky-Sly}). The details are given in subsection \ref{3.3}.
\begin{cor}\label{Glaub-lower}
For the continuous time Glauber dynamics where each vertex is updated at rate 1:
\begin{enumerate}
\item 
On the random $d$-regular graph graph with $\beta=\beta_c$, there exists some constant $c_0>0$ such that with probability $1-o(1)$, the mixing time of the Glauber dynamics satisfies $t_{mix}\ge c_0\sqrt{n}$.
\item 
On the random Erd\"os-R\'enyi random graph $G(n,d/n)$, at $\beta=\beta_c$, for every $\varepsilon>0$ there exists $\delta>0$ such that $\mathbb{P}[t_{mix}\geq \delta\sqrt{n}]\geq 1-\varepsilon$.
\end{enumerate}
\end{cor}
Recently, Coja-Oghlan et. al.~\cite{Coja-et-al}  studied the fluctuations of the free energy of the Ising model in the high and low-temperature regimes and asked whether it is possible to study the fluctuations at criticality. Our next corollary answers this question:
\begin{cor}\label{free-evergy-ER}
Consider the Ising model on $G(n,d/n)$ at the critical temperature and let
$$Z_n:=\SL{\sigma\in\{-1,1\}^n}\exp\left(\beta\SL{u\sim v}\sigma_u\sigma_v\right)$$
be its partition function. If we set
\begin{align*}
\Delta_n:=\log(Z_n)-n\log(2)-|E_n|\log\cosh(\beta)-\log\dfrac{\sqrt[4]{n}}{\sqrt{2\pi}}+\dfrac34,
\end{align*}
where $|E_n|$ is the number of edges of the graph, then
\begin{align*}
\Delta_n\xrightarrow[n\to\infty]{(d)}&\SL{i=3}^\infty\left(\tilde{C}_i\log(1+d^{-i})-\frac{d^i+1}{2i}\right)\\&+\log\int_\R\exp\left(\frac{y^2}{\sqrt{2(d-1)}}X-\left(\dfrac{1}{12}+\dfrac{1}{4(d-1)}\right)y^4\right)\ \text{d}y
\end{align*}
where the $\tilde{C}_i\sim\text{Pois}(\frac{d^i}{2i})$ are independent, and also independent of $X\sim N(0,1)$.
\end{cor}
\begin{rem}
Let $Y$ be the limit described above. It is also true that
$$\left(\Delta_n,\frac{|E_n|-\frac{1}{2}dn}{\sqrt{\frac12dn}}\right)\xrightarrow[n\to\infty]{(d)}(Y,X_1),$$
where $(X,X_1)$ is a Gaussian vector, independent of the $\tilde{C}_i$, with $X,X_1\sim N(0,1)$ and $\text{Cov}(X,X_1)=\sqrt{\frac{d-1}{d}}$. Indeed in this case, $X_1$ is the limit of the normalized edge count and $X$ is the limit of the normalized of the long path count, and one could prove that $(X,X_1)$ has the above properties, with a calculation similar to that performed in the proof of Proposition \ref{cyc-path-ER}. However, to keep the paper to a manageable length, we do not pursue that here.
\end{rem}

\subsection*{Acknowledgments}  We would like to thank Nick Wormald and Eyal Lubetzky for helpful discussions.  The work was partially supported by a Simons Investigator grant.

\section{Preliminaries}\label{2}
In the entirety of this paper, whenever we write $a_n\sim b_n$ for two sequences $a_n,b_n$, we mean that $a_n/b_n\to1$ as $n\to\infty$. Also, $A\asymp B$ will mean that $c_1B\le A\le c_2B$ for some absolute constants $c_1,c_2>0$.
\subsection{Definitions}
We start by introducing the Ising model and the Glauber dynamics.
\begin{defn}
\begin{itemize}
\item 
Let $G$ be a graph with vertex set $V$. The Ising model on this graph with inverse temperature $\beta$ is the probability measure $\mu$ on $\{-1,1\}^V$ satisfying
$$\mu_{\beta}(\sigma)=\dfrac{1}{Z_{\beta,G}}\exp\left(\beta\SL{u\sim v}\sigma_u\sigma_v\right).$$
\item
Fix a measure $\nu$ on $\{-1,1\}^V$. For a configuration $\sigma\in\{-1,1\}^V$ and for some $v\in V$, let $\sigma^{\oplus v}$ be the configuration $\sigma$ in which the spin at vertex $v$ is flipped. For two configurations $\sigma,\tau\in\{-1,1\}^V$, we write $\sigma\sim\tau$ if $\tau=\sigma^{\oplus v}$ for some $v\in V$. The continuous time Glauber Dynamics on $G$ for $\nu$ at rate 1 is a Markov chain with transition rates, for $\sigma\neq\tau$, equal to
$$q(\sigma,\tau)=\begin{cases}
\dfrac{\nu(\tau)}{\nu(\sigma)+\nu(\tau)},\ \ \text{if}\ \ \sigma \sim \tau
\\ \\
0,\ \ \text{otherwise}.
\end{cases}$$
Informally, every vertex gets updated at rate one, and the update happens conditionally on the configuration in the rest of the vertices.
\end{itemize}
\end{defn}
We continue with the $p$-Wasserstein distance. The special $p=1$ case is the metric used in Theorems \ref{main-thm-1}, \ref{main-thm-2}. For a metric space $(X,\text{d})$, we denote by $\mathcal{M}_1(X)$ to be the set of probability measures on $X$.
\begin{defn}
The $p$-Wasserstein distance associated to $\text{d}$ is the function
$$W_p(\mu,\nu):=\inf\limits_{\gamma\in\mathcal{C}(\mu,\nu)}\left(\E_{(x,y)\sim\gamma}\text{d}(x,y)^p\right)^{1/p},$$
where the $\inf$ is taken over all couplings $\gamma$ of $\mu$ and $\nu$.
\end{defn}
\subsection{Moments of partition functions in regular graphs}
Let $G$ be a $d$-regular graph on $n$ vertices. We will call a measure $\mu$ on $\mathcal{X}^V$ a homogeneous \textit{spin system} if it is of the form
$$\mu(\sigma)\varpropto\PL{u\in V}\bar{\psi}(\sigma_u)\PL{u\sim v}\psi(\sigma_u,\sigma_v),$$
where $\psi:\mathcal{X}^2\to[0,\infty)$ and $\bar{\psi}:\mathcal{X}\to[0,\infty)$ are functions, called the weights of the spin system. Suppose we are given a finite spin set $\mathcal{X}$ and weights $(\overline{\psi},\psi)$ with $\psi>0$. For a configuration $\sigma\in\mathcal{X}^V$, its edge-empirical distribution is
$$h_\sigma(x,x'):=\dfrac{1}{dn}\SL{u\in V}\SL{v\in\partial u}\mathbf{1}_{\sigma_u=x,\ \sigma_v=x'}.$$
Suppose $G$ is a $d$-regular graph. Then, the edge-empirical distribution uniquely determines the vertex-empirical distribution as well:
$$\overline{h}_\sigma(x):=\dfrac{1}{n}\SL{u\in V}\mathbf{1}_{\sigma_u=x}=\SL{x'\in\mathcal{X}}h_\sigma(x,x').$$
Let $\mathcal{H}_n$ be the set of $h\in[0,1]^{\mathcal{X}^2}$ that are possible edge-empirical distributions (for some graph $G$) and $\overline{\mathcal{H}}_n$ be the set of $h\in[0,1]^\mathcal{X}$ that are possible vertex-empirical distributions. Also, for any $h\in\mathcal{H}_n$ let $A_h$ be the set of configurations that have edge-empirical density $h$. Then, under the $d$-regular configuration model,
\begin{align}\label{A-form-reg}
\nonumber\E|A_h|=\dfrac{1}{(dn-1)!!}&\cdot\dfrac{n!}{\PL{x\in\mathcal{X}}(n\overline{h}(x))!}\cdot\PL{x\in\mathcal{X}}\dfrac{(dn\overline{h}(x))!}{\PL{x'\in\mathcal{X}}(dnh(x,x'))!}\\&\cdot\PL{x\in\mathcal{X}}\left((dnh(x,x)-1)!!\cdot\PL{x'\neq x}\sqrt{(dnh(x,x'))!}\right).
\end{align}
The proof we present here can be found, for example, as an intermediate step for proving Theorem 3 in~\cite{D-M-S-S}. If we fix an edge-empirical distribution, there are $\frac{n!}{\PL{x\in\mathcal{X}}(n\overline{h}(x))!}$ ways to choose the spins of the vertices. Then, there are $\PL{x\in\mathcal{X}}\frac{(dn\overline{h}(x))!}{\PL{x'\in\mathcal{X}}(dnh(x,x'))!}$ ways to choose the spin at the half-edges and then another $\PL{x\in\mathcal{X}}\left((dnh(x,x)-1)!!\cdot\PL{x'\neq x}\sqrt{(dnh(x,x'))!}\right)$ admissible ways to find a matching of the half-edges.\\\\ 
We will use the fact that if
$$Z_h=\SL{\sigma\in A_h}\ \PL{u\in V}\overline{\psi}(\sigma_u)\PL{\{u,v\}\in E}\psi(\sigma_u,\sigma_v)$$
is the contribution of configurations $\sigma$ with $h(\sigma)=h$ to the partition function, then for every $h\in\mathcal{H}_n$,
\begin{equation}\label{h-contr-reg}
\E(Z_h)=\exp\left(n\left[\langle\overline{h},\log\overline{\psi}\rangle+\frac{d}{2}\langle h,\log\psi\rangle\right]\right)\cdot\E|A_h|.
\end{equation}
It is important to note that due to Stirling's formula, due to (\ref{A-form-reg}) for any $h$, $\E(Z_h)\asymp n^{\Theta(1)}\cdot\exp(n\Phi(h)),$ where
\begin{equation}\label{Phi-defn}
\Phi(h)=\dfrac{d}{2}H(h)-(d-1)H(\bar{h})+\dfrac{d}{2}\langle h,\log\psi\rangle+\langle\bar{h},\log\bar{\psi}\rangle.
\end{equation}
Here, we denote by $H(p)$ the Shannon entropy of a probability vector $p$. Therefore, the $h\in\mathcal{H}_n$ that are close to the maximizer of $\Phi$ are naturally the ones whose contributions dominate the expectation of the partition function. This intuition will be formalized and used repeatedly in sections \ref{4} and \ref{5}.
\section{Proofs of main Theorems}
\subsection{Proof of Theorem \ref{main-thm-1}}
In this subsection, we prove Theorem \ref{main-thm-1}. We will work with the $d$-regular configuration model on $n$ vertices instead. Since the probability of simplicity of the configuration model converges to a positive constant, proving Theorem \ref{main-thm-1} for the configuration model implies that it holds for $G_{n,d}$ as well.\\\\ 
At first, we state a useful lemma which will later imply that $n^{3/4}$ is the right scaling for the magnetization. Its proof will be in Section \ref{4}.
\begin{lem}\label{right-scal-reg}
For every $\varepsilon>0$, there exists $C=C(\varepsilon)>0$, such that
$$\PR\left[\SL{\sigma:|m(\sigma)|>Cn^{3/4}}Z_\sigma\ge\dfrac{1}{2}\sqrt{\dfrac{d-2}{d-1}}\varepsilon^2\cdot\E\left(\SL{\sigma:|m(\sigma)|\le Cn^{3/4}}Z_\sigma\right)\right]\le\varepsilon.$$
\end{lem}
To prove Theorem \ref{main-thm-1}, we will need a relation between the second moment and the first moment of $Z_m$ and the distribution of cycle counts under the regular and the planted measure.
\begin{lem}\label{Zm-reg-mom} Let $C>0$ and $m=m(n)$ be a sequence of positive integers with the same parity as $n$, with $|m|\le Cn^{3/4}$. Then,
$$\E(Z_m)\sim c_{d,n}\cdot\exp\left(-\dfrac{(d-1)(d-2)}{12d^2}\cdot \dfrac{m^4}{n^3}\right),$$
where 
$$c_{d,n}=\dfrac{2^{1/2}\sqrt{d-1}}{\sqrt{n}\sqrt{\pi d}}\cdot(2\cdot\cosh(\beta)^{d/2})^n,$$ 
and
$$\dfrac{\E(Z_m^2)}{\E(Z_m)^2}\xrightarrow[n\to\infty]{}\sqrt{\dfrac{d-1}{d-2}}.$$
Moreover, both of these convergences are uniform in $m$.
\end{lem}
\vspace{1em}
For each $i\ge3$, let $Y_{i,n}$ be the number of cycles of length $i$ present in $G$. For $i=2$, $Y_{2,n}$ is the number of double edges in $G$, whereas $Y_{1,n}$ is the number of loops in $G$. Regarding the distribution of $Y_{i,n}$, we prove the following proposition:
\begin{prop}\label{cyc-reg}
\begin{enumerate}
\item 
For each $i\ge1$,
$$Y_{i,n}\xrightarrow[n\to\infty]{(d)}C_i\sim\text{Pois}\left(\dfrac{(d-1)^i}{2i}\right).$$
These convergences hold jointly, and the limits $C_i$ are independent.
\item
Let $m$ be as in Lemma \ref{Zm-reg-mom}. Under the planted measure $\PR_m^*$, for any $i\ge1$,
$$Y_{i,n}\xrightarrow[n\to\infty]{(d)}C_i'\sim\text{Pois}\left(\dfrac{(d-1)^i+1}{2i}\right).$$
As before, these convergences hold jointly, and the limits are independent.
\end{enumerate}
\end{prop}
In view of Proposition \ref{cyc-reg}, we set $\delta_i=(d-1)^{-i}$, $\lambda_i=\frac{(d-1)^i}{2i}$ and for any $k\in\N$,
\begin{align*}
W_{k,n}=&\SL{i=1}^k\left(Y_{i,n}\log(1+\delta_i)-\lambda_i\delta_i\right),\ \ W_{k,n}^{(R)}=\sgn(W_{k,n})\min(|W_{k,n}|,R)\\
W_k=&\SL{i=1}^k\left(C_i\log(1+\delta_i)-\lambda_i\delta_i\right),\ \  W_k^{(R)}=\sgn(W_k)\min(|W_k|,R),\\ W_k'=&\SL{i=1}^k\left(C_i'\log(1+\delta_i)-\lambda_i\delta_i\right),\ \ W_k'^{(R)}=\sgn(W_k')\min(|W_k'|,R)\ \ \text{and}\\
W_\infty=&\SL{i=1}^\infty(C_i\log(1+\delta_i)-\lambda_i\delta_i),\ \ \text{and}\ \ W_\infty'=\SL{i=1}^\infty(C_i'\log(1+\delta_i)-\lambda_i\delta_i).
\end{align*}
\begin{rem}
\begin{itemize}
\item
For any $\delta_0\in(0,1)$ and $k\in\N$,
\begin{equation}\label{Wk-size}
\PR\left(e^{W_k}\le\delta_0\right)=\PR\left(e^{-W_k}\ge\delta_0^{-1}\right)\le\delta_0\cdot\E(e^{-W_k})=\delta_0\cdot\exp\left(\SL{i=1}^k\frac{\lambda_i\delta_i^2}{1+\delta_i}\right)\le\delta_0\cdot\sqrt{\dfrac{d-1}{d-2}}.
\end{equation}
\item 
Since $W_k^{(R)}\ge W_k$ when $W_k<0$, the same inequality holds for $W_k^{(R)}$ as well.
\item 
$\exp(2W_k)\xrightarrow[k\to\infty]{}\exp(2W_\infty)$ and $\exp(W_k')\xrightarrow[k\to\infty]{}\exp(W_\infty')$, both a.s. and in $L^2$, as was noted in \cite{Janson1995}.
\end{itemize}
\end{rem}
Our main claim is the following:
\begin{prop}\label{Zm-appr-reg}
For any $k\in\N$ and $C>0$,
$$\limsup\limits_{n\to\infty}\sup\limits_{|m|\le Cn^{3/4}}\E\left[\left(\dfrac{Z_m}{\E(Z_m)}-\exp\left(W_{k,n}^{(R)}\right)\right)^2\right]\le\alpha_{k,R,C}$$
where
$$\lim\limits_{R\to\infty}\lim\limits_{k\to\infty}\alpha_{k,R,C}=0$$
for any $C>0$.
\end{prop}
We explain why Proposition \ref{Zm-appr-reg} is enough to finish the proof of Theorem \ref{main-thm-1}.
\begin{proof}[Proof of Theorem \ref{main-thm-1}]
Let $\varepsilon\in(0,1)$. Choose $C>0$ large enough, so that the statement of Lemma \ref{right-scal-reg} holds and so that $\mu_\text{reg}([-C,C])>1-\varepsilon$. Also, let $-C=y_0<y_1<\cdots<y_s=C$ be real numbers so that for any $i=0,1,\dots,s-1$,
$$\mu_{\text{reg}}([y_i,y_{i+1}])<\varepsilon.$$
By Skorokhod's representation Theorem, it is possible couple the relevant random variables, so that the convergences in Proposition \ref{cyc-reg} are almost sure convergences. Note that the coupling implied in the Wasserstein distance in this case is trivialized, since $\mu_{\text{reg}}$ is deterministic. Choose $\delta_1,\delta_2,\delta>0$ such that: 
\begin{equation}\label{delta-reg}
\PR(\exp(W_k)\le\delta_1)\le\varepsilon,\ \  \frac{2\delta_2}{\delta_1-\delta_2}\le\varepsilon,\ \ \text{and}\ \ \frac{s\delta}{\delta_2^2}\le\varepsilon.
\end{equation}
Due to (\ref{Wk-size}), it is possible to set $\delta_1=\sqrt{\frac{d-2}{d-1}}\cdot\varepsilon$. Since $\varepsilon\in(0,1)$, $\delta_2<\delta_1/2$. Also, for each $i\in\{1,2,\dots,s\}$, let
$$Z^{(i)}=\SL{-Cn^{3/4}\le m\le y_i n^{3/4}}Z_m.$$
Because of Proposition \ref{Zm-reg-mom}, we can set
$$\sup\limits_n\dfrac{\sup\limits_{|m|\le Cn^{3/4}}\E(Z_m)}{\inf\limits_{|m|\le Cn^{3/4}}\E(Z_m)}=:M<\infty.$$
For any $i$, set $N_i=\#\{m:-Cn^{3/4}\le m\le y_in^{3/4}\}$. Due to Proposition \ref{Zm-appr-reg}, for every $i$, if $n$ is large enough,
\begin{align*}
\E\left[\left(Z^{(i)}-\exp\left(W_{k,n}^{(R)}\right)\E(Z^{(i)})\right)^2\right]&=\E\left[\left(\SL{-Cn^{3/4}\le m\le y_in^{3/4}}(Z_m-\exp\left(W_{k,n}^{(R)}\right)\E(Z_m))\right)^2\right]
\\&\le N_i^2\cdot\sup\limits_{-Cn^{3/4}\le m\le y_in^{3/4}}\E\left[\left(Z_m-\exp(W_{k,n}^{(R)})\E(Z_m)\right)^2\right]\\&\le N_i^2\cdot2\alpha_{k,R,C}\cdot\left(\sup\limits_{-Cn^{3/4}\le m\le y_in^{3/4}}\E(Z_m)\right)^2\\&\le2M^2\alpha_{k,R,C}\left(\E(Z^{(i)})\right)^2.
\end{align*}
Therefore, if $k,R$ and $n$ are large enough,
\begin{equation}\label{Zi-appr-reg}
\E\left[\left(\dfrac{Z^{(i)}}{\E(Z^{(i)})}-\exp\left(W_{k,n}^{(R)}\right)\right)^2\right]\le\delta.
\end{equation}
Consider the event
$$A_n^{(1)}:=\left\{\exp\left(W_{k,n}^{(R)}\right)>\delta_1\right\}\cap\bigcap\limits_{i=1}^s\left\{\left|\dfrac{Z^{(i)}}{\E(Z^{(i)})}-\exp\left(W_{k,n}^{(R)}\right)\right|\le\delta_2\right\}.$$
On the event $A_n^{(1)}$,
$$\dfrac{Z^{(s)}}{\E(Z^{(s)})}\ge\delta_1-\delta_2\ge\dfrac{1}{2}\sqrt{\dfrac{d-2}{d-1}}\varepsilon.$$
Therefore, on the event
$$A_n=A_n^{(1)}\cap\left\{\SL{\sigma:|m(\sigma)|>Cn^{3/4}}Z_\sigma\le\dfrac{1}{2}\sqrt{\dfrac{d-2}{d-1}}\varepsilon^2\cdot\E\left(Z^{(s)}\right)\right\},$$
it is true that
$$\dfrac{\SL{\sigma:|m(\sigma)|>Cn^{3/4}}Z_\sigma}{Z}<\dfrac{\SL{\sigma:|m(\sigma)|>Cn^{3/4}}Z_\sigma}{Z^{(s)}}\le\varepsilon\ \ \Rightarrow\ \dfrac{Z-Z^{(s)}}{Z}<\varepsilon.$$
We prove that for large $n$, $\PR(A_n)\ge1-4\varepsilon$. Due to (\ref{Zi-appr-reg}) and Chebyshev's inequality,
$$\PR\left(\bigcup\limits_{i=1}^s\left\{\left|\dfrac{Z^{(i)}}{\E(Z^{(i)})}-\exp\left(W_{k,n}^{(R)}\right)\right|>\delta_2\right\}\right)\le s\cdot\dfrac{\delta}{\delta_2^2}\le\varepsilon.$$
Also, by our choice of $\delta_1$, if $n$ is large enough,
$$\PR\left(\exp\left(W_{k,n}^{(R)}\right)\le\delta_1\right)\le2\varepsilon,$$
which indeed implies that $\PR(A_n^c)\le4\varepsilon$.\\\\
Our claim is that on the event $A_n$,
\begin{equation}\label{conv-prob-reg}
\text{d}_{\text{KS}}(\mu_{n,\text{reg}},\mu_{\text{reg}})<10\varepsilon.
\end{equation}
This will imply the result, as
$$\E[\text{d}_{\text{KS}}(\mu_{n,\text{reg}},\mu_{\text{reg}})]\le10\varepsilon+\PR(A_n^c)<14\varepsilon$$
if $n$ is large enough, since $\text{d}_{\text{KS}}(\mu,\nu)\le1$ for any measures $\mu,\nu$ due to the definition of $\text{d}_{\text{KS}}$. We first prove that on the event $A_n$,
\begin{equation}\label{conv-y_i-reg}
\sup\limits_{1\le i\le s}|\mu_{n,\text{reg}}([-C,y_i])-\mu_{\text{reg}}([-C,y_i])|<4\varepsilon.
\end{equation}
Indeed,
\begin{align*}
&|\mu_{n,\text{reg}}([-C,y_i])-\mu_{\text{reg}}([-C,y_i])|\\=\ & \left|\dfrac{Z^{(i)}}{Z}-\dfrac{\int_{[-C,y_i]}\exp\left(-\frac{(d-1)(d-2)}{12d^2}u^4\right)\ \text{d}u}{\int_\R\exp\left(-\frac{(d-1)(d-2)}{12d^2}u^4\right)\ \text{d}u}\right|
\\\le\ & \left|\dfrac{Z^{(i)}}{Z}-\dfrac{Z^{(i)}}{Z^{(s)}}\right|+\left|\dfrac{\int_{[-C,y_i]}\exp\left(-\frac{(d-1)(d-2)}{12d^2}u^4\right)\ \text{d}u}{\int_{[-C,C]}\exp\left(-\frac{(d-1)(d-2)}{12d^2}u^4\right)\ \text{d}u}-\dfrac{\int_{[-C,y_i]}\exp\left(-\frac{(d-1)(d-2)}{12d^2}u^4\right)\ \text{d}u}{\int_\R\exp\left(-\frac{(d-1)(d-2)}{12d^2}u^4\right)\ \text{d}u}\right|\\+&\left|\dfrac{Z^{(i)}}{Z^{(s)}}-\dfrac{\E(Z^{(i)})}{\E(Z^{(s)})}\right|+\left|\dfrac{\E(Z^{(i)})}{\E(Z^{(s)})}-\dfrac{\int_{[-C,y_i]}\exp\left(-\frac{(d-1)(d-2)}{12d^2}u^4\right)\ \text{d}u}{\int_{[-C,C]}\exp\left(-\frac{(d-1)(d-2)}{12d^2}u^4\right)\ \text{d}u}\right|.
\end{align*}
We show that on the event $A_n$, each one of these terms is $<\varepsilon$.
\begin{enumerate}
\item 
Since $\frac{Z-Z^{(s)}}{Z}<\varepsilon$, and due to the way $C$ was chosen, the first two terms are $<\varepsilon$.
\item 
When $n\to\infty$, due to Lemma \ref{Zm-reg-mom}, the last term is deterministic and converges to 0:
\begin{align*}
\E(Z_i)&\sim c_{d,n}\SL{-Cn^{3/4}\le m\le y_in^{3/4}}\exp\left(-\dfrac{(d-1)(d-2)}{12d^2}\cdot\frac{m^4}{n^3}\right)
\\&\sim c_{d,n}\frac{n^{3/4}}{2}\int_{-C}^{y_i}\exp\left(-\frac{(d-1)(d-2)}{12d^2}u^4\right)\ \text{d}u,
\end{align*}
because of the approximation of the integral by a Riemann sum.
\item 
On the event $A_n$, $Z^{(s)}\ge(\delta_1-\delta_2)\cdot\E(Z^{(s)})$ and 
$$\left|\dfrac{Z^{(i)}}{\E(Z^{(i)})}-\dfrac{Z^{(s)}}{\E(Z^{(s)})}\right|\le\left|\dfrac{Z^{(i)}}{\E(Z^{(i)})}-\exp\left(W_{k,n}^{(R)}\right)\right|+\left|\dfrac{Z^{(s)}}{\E(Z^{(s)})}-\exp\left(W_{k,n}^{(R)}\right)\right|\le2\delta_2,$$
which implies
$$\left|\dfrac{Z^{(i)}}{Z^{(s)}}-\dfrac{\E(Z^{(i)})}{\E(Z^{(s)})}\right|\le\left|\dfrac{Z^{(i)}}{\E(Z^{(i)})}-\dfrac{Z^{(s)}}{\E(Z^{(s)})}\right|\cdot\dfrac{\E(Z^{(s)})}{Z^{(s)}}\le\dfrac{2\delta_2}{\delta_1-\delta_2}\le\varepsilon.$$
\end{enumerate}
We have successfully proven (\ref{conv-y_i-reg}). Let $y\in\R$. If $|y|>C$, due to the way $C$ was chosen, if $n$ is large enough, $|\mu_{n,\text{reg}}((-\infty,y])-\mu_{\text{reg}}((-\infty,y])|<\varepsilon$. On the other hand, if $|y|\le C$, 
$$|\mu_{n,\text{reg}}((-\infty,y])-\mu_{\text{reg}}((-\infty,y])|\le\varepsilon+|\mu_{n,\text{reg}}([-C,y])-\mu_{\text{reg}}([-C,y])|$$
and there exists $0\le i\le s-1$ such that $y_i\le y\le y_{i+1}$. For this $i$, we claim that $\mu_{n,\text{reg}}([-C,y])$ and $\mu_{\text{reg}}([-C,y])$ lie in the interval $I=[\mu_{\text{reg}}([-C,y_i])-4\varepsilon,\mu_{\text{reg}}([-C,y_{i+1}])+4\varepsilon]$. This will mean that their difference is at most the length of the interval, which is $\le9\varepsilon$, concluding the proof. This fact is obvious for $\mu_{\text{reg}}([-C,y])$. As for $\mu_{n,\text{reg}}([-C,y])$, observe that on the event $A_n$, if $n$ is large enough,
\begin{align*}
\mu_{\text{reg}}([-C,y_i])-4\varepsilon\le\mu_{n,\text{reg}}([-C,y_i])&\le\mu_{n,\text{reg}}([-C,y])\\&\le\mu_{n,\text{reg}}([-C,y_{i+1}])\le\mu_{\text{reg}}([-C,y_{i+1}])+4\varepsilon.
\end{align*}
This completes the proof of Theorem \ref{main-thm-1}.
\end{proof}
It, therefore, remains to show Proposition \ref{Zm-appr-reg}, given Lemma \ref{Zm-reg-mom} and Proposition \ref{cyc-reg}.
\begin{proof}[Proof of Proposition \ref{Zm-appr-reg}]
We show the Proposition with
$$\alpha_{k,R,C}:=\sqrt{\dfrac{d-1}{d-2}}-2\cdot\E\left[\exp\left(W_k'^{(R)}\right)\right]+\E\left[\exp\left(2W_k^{(R)}\right)\right].$$
Due to the dominated convergence theorem and the integrability of $\exp(W_\infty'), \exp(2W_\infty)$,
\begin{align*}
\lim\limits_{R\to\infty}\lim\limits_{k\to\infty}\E\left[\exp\left(W_k'^{(R)}\right)\right]&=\PL{i=1}^\infty e^{-\lambda_i\delta_i}\E\left((1+\delta_i)^{C_i'}\right)=\PL{i=1}^\infty\exp\left(\lambda_i\delta_i^2\right)=\sqrt{\dfrac{d-1}{d-2}}\ \ \text{and}
\\\lim\limits_{R\to\infty}\lim\limits_{k\to\infty}\E\left[\exp\left(2W_k^{(R)}\right)\right]&=\PL{i=1}^\infty e^{-2\lambda_i\delta_i}\E\left((1+\delta_i)^{2C_i}\right)=\PL{i=1}^\infty\exp(\lambda_i\delta_i^2)=\sqrt{\dfrac{d-1}{d-2}},
\end{align*}
so the condition claimed about the $\alpha_{k,R,C}$ holds.\\\\
For any sequence of $m=m(n)$ as in Lemma \ref{Zm-reg-mom},
\begin{align*}
\E\left[\left(\dfrac{Z_m}{\E(Z_m)}-\exp\left(W_{k,n}^{(R)}\right)\right)^2\right]=&\ \dfrac{\E(Z_m^2)}{\E(Z_m)^2}-2\cdot\E_m^*\left(\exp\left(W_{k,n}^{(R)}\right)\right)+\E\left(\exp\left(2W_{k,n}^{(R)}\right)\right)\\\underset{n\to\infty}{\to}&\ \alpha_{k,R,C}
\end{align*}
which implies the result.
\end{proof}
\subsection{Proof of Theorem \ref{main-thm-2}}
In this subsection, we complete the proof of Theorem \ref{main-thm-2}. In this case, we study a normalized version of the partition function. This is necessary; without it, proving concentration is impossible, due to, for example, the fluctuations in the number of edges. For $\sigma\in\{-1,1\}^n$, let
$$\tilde{Z}_\sigma=2^{-n}\PL{1\le u<v\le n}\dfrac{\exp(\beta\sigma_u\sigma_v\cdot\mathbf{1}_{u\sim v})}{\cosh(\beta)^{\mathbf{1}_{u\sim v}}}.$$
Naturally, for any $m$ that has the same parity as $n$, we set
$$\tilde{Z}_m:=\SL{\sigma:m(\sigma)=m}\tilde{Z}_\sigma\ \ \text{and}\ \ \tilde{Z}=\SL{m}\tilde{Z}_m.$$
Also, we denote by $\PR_m^*$ the planted measure on the set of graphs, defined by $\PR_m^*(A)=\frac{\E(\tilde{Z}_m\mathbf{1}_A)}{\E(\tilde{Z}_m)}$.\\\\
At first, we prove the analog of the initial Lemma \ref{right-scal-reg} in the Erd\"os-R\'enyi case.
\begin{lem}\label{right-scal-ER}
Set
$$\tilde{c}_0=\inf\limits_{C>0}\liminf\limits_{n\to\infty}\dfrac{\SL{|m|\le n^{3/4}}\E(\tilde{Z}_m)}{\SL{|m|\le Cn^{3/4}}\E(\tilde{Z}_m)}.$$
Then, $\tilde{c}_0>0$ and for any $\varepsilon>0$, there exists $C=C(\varepsilon)>0$ such that
$$\PR\left[\SL{|m|>Cn^{3/4}}\tilde{Z}_m\ge\dfrac{\tilde{c}_0}{4}\sqrt{\dfrac{d}{d-1}}\varepsilon^2\exp\left(-\dfrac{\varepsilon^{-1}}{\sqrt{2(d-1)}}-\dfrac{1}{4(d-1)}\right)\cdot\E\left(\SL{|m|\le Cn^{3/4}}\tilde{Z}_m\right)\right]<\varepsilon.$$
\end{lem}
As in the previous subsection, we compute the first and second moments of $\tilde{Z}_m$.
\begin{lem}\label{Zm-ER-mom}
Let $m=m(n)$ be a sequence of positive integers with the same parity as $n$, with $|m|\le Cn^{3/4}$. Then,
$$\E(\tilde{Z}_m)\sim\dfrac{\sqrt{2}}{\sqrt{\pi n}}\cdot\exp\left(-\dfrac{1}{12}\cdot\dfrac{m^4}{n^3}-\dfrac{3}{4}\right)$$
and
$$\dfrac{\E(\tilde{Z}_m^2)}{\E(\tilde{Z}_m)^2}\sim\exp\left(\dfrac{1}{2(d-1)}\cdot\dfrac{m^4}{n^3}+\SL{i=3}^\infty\dfrac{d^{-i}}{2i}\right).$$
Moreover, these convergences are uniform in $m$.
\end{lem}
Let $X_{\ell,n}$ denote the number of paths of length $\ell$ present in $G$ and 
\begin{equation*}
\widehat{X}_{\ell,n}:=\dfrac{X_{\ell,n}-\frac{1}{2}nd^\ell}{\sqrt{\frac12nd^{2\ell}\frac{\ell^2}{d-1}}},
\end{equation*}
as introduced in (\ref{hat-X-def}). 
Moreover, for any $R>0$, set $\widehat{X}_{\ell,n}^{(R)}=\sgn(\widehat{X}_{\ell,n})\min(|\widehat{X}_{\ell,n}^{(R)}|,R)$.\\\\
For each $i\ge3$, let $\tilde{Y}_{i,n}$ be the number of cycles of length $i$ present in $G$. The second ingredient we will need is the asymptotic distribution of $\widehat{X}_{\ell,n}$ and $\tilde{Y}_{i,n}$ both under the regular and the planted measure.
\begin{prop}\label{cyc-path-ER}
\begin{enumerate}
\item\label{pois-cycle-ER}
For each $i\ge3$,
$$\tilde{Y}_{i,n}\xrightarrow[n\to\infty]{(d)}\tilde{C}_i\sim\text{Pois}\left(\frac{d^i}{2i}\right).$$
These convergences hold jointly, and the limits $C_i$ are independent.
\item 
For any $\ell\in\N$,
$$\widehat{X}_{\ell,n}\xrightarrow[n\to\infty]{(d)}X_\ell\sim N\left(0,1+\gamma_\ell^{(2)}\right),$$
where $$\gamma_\ell^{(2)}:=\dfrac{d-1}{\ell^2}\SL{k=1}^\ell\left[(\ell-k+1)^2\cdot d^{-k}\right]-1\xrightarrow[\ell\to\infty]{}0.$$ 
Also, this convergence holds jointly with the convergences in \ref{pois-cycle-ER} and the limits are independent.
\item\label{pois-cycle-pl-ER}
Let $m$ be as in Lemma \ref{Zm-ER-mom}. Under the planted measure $\PR_m^*$, for any $i\ge3$
$$\tilde{Y}_{i,n}\xrightarrow[n\to\infty]{(d)}\tilde{C}_i'\sim\text{Pois}\left(\frac{d^i+1}{2i}\right).$$
As before, these convergences hold jointly, and the limits are independent.
\item 
Let $m$ be as in Lemma \ref{Zm-ER-mom}, moreover satisfying the condition that $m\cdot n^{-3/4}\xrightarrow[n\to\infty]{}x$. Under the planted measure $\PR_m^*$, for any $\ell\in\N$
$$\widehat{X}_{\ell,n}\xrightarrow[n\to\infty]{(d)}X_\ell'\sim N\left(\dfrac{x^2}{\sqrt{2(d-1)}}\left(1+\gamma_\ell^{(1)}\right),1+\gamma_\ell^{(2)}\right),$$
where
$$\gamma_\ell^{(1)}:=\dfrac{d-1}{\ell}\SL{k=1}^\ell\left[(\ell-k+1)\cdot d^{-k}\right]-1\xrightarrow[\ell\to\infty]{}0.$$
Also, this convergence holds jointly with the convergences in \ref{pois-cycle-pl-ER} and the limits are independent.
\end{enumerate}
\end{prop}
In view of \ref{pois-cycle-pl-ER}, we set $\tilde{\delta}_i=d^{-i}$, $\tilde{\lambda}_i=\frac{d^i}{2i}$ and $\tilde{\mu}_i=\tilde{\lambda}_i(1+\tilde{\delta}_i)$. Also, just as in the proof of Theorem \ref{main-thm-1}, let
\begin{align*}
\tilde{W}_{k,n}:=&\SL{i=3}^k\left(\tilde{Y}_{i,n}\log(1+\tilde{\delta}_i)-\tilde{\lambda}_i\tilde{\delta}_i\right),\ \ \tilde{W}_{k,n}^{(R)}=\sgn(\tilde{W}_{k,n})\min(|\tilde{W}_{k,n}|,R)\\
\tilde{W}_k=&\SL{i=3}^k\left(\tilde{C}_i\log(1+\tilde{\delta}_i)-\tilde{\lambda}_i\tilde{\delta}_i\right),\ \  \tilde{W}_{k}^{(R)}=\sgn(\tilde{W}_{k})\min(|\tilde{W}_{k}|,R)\ \ \text{and}\\ \tilde{W}_k'=&\SL{i=3}^k\left(\tilde{C}_i'\log(1+\tilde{\delta}_i)-\tilde{\lambda}_i\tilde{\delta}_i\right)\ \ \text{and}\ \ \tilde{W}_k'^{(R)}=\sgn(\tilde{W}_{k}')\min(|\tilde{W}_{k}'|,R).
\end{align*}
Our main claim is the following:
\begin{prop}\label{Zm-appr-ER}
For any $\ell,k\in\N$ and $R,C>0$,
$$\limsup\limits_{n\to\infty}\sup\limits_{|m|\le Cn^{3/4}}\E\left[\left(\dfrac{\tilde{Z}_m}{\E(\tilde{Z}_m)}-\exp\left(\theta_m\widehat{X}_{\ell,n}^{(R)}-\theta_m^2/2+\tilde{W}_{k,n}^{(R)}\right)\right)^2\right]\le\tilde{\alpha}_{\ell,k,R,C}$$
where $$\theta_m=\dfrac{(m\cdot n^{-3/4})^2}{\sqrt{2(d-1)}}\ \ \text{and}\ \ \lim\limits_{R\to\infty}\lim\limits_{\ell\to\infty}\lim\limits_{k\to\infty}\tilde{\alpha}_{\ell,k,R,C}=0$$ 
for any $C>0$.
\end{prop}
We explain how this Proposition implies Theorem \ref{main-thm-2}.
\begin{proof}[Proof of Theorem \ref{main-thm-2}]
Due to the Skorokhod's representation Theorem, we may couple all the graphs and random variables such that the convergences in Proposition \ref{cyc-path-ER} are almost sure convergences. We claim that if $\ell,R$ and $n$ are large enough, then
$$\text{d}_{\text{KS}}\left(\mu_{n,\text{ER}},\mu^{(X_\ell^{(R)})}\right)\le18\varepsilon.$$
This will imply the desired result: If $\ell$ and $R$ are large enough, we can couple $X_\ell^{(R)}$ with a $X\sim N(0,1)$ random variable such that $X_\ell^{(R)}=X$ with probability $\ge1-\varepsilon$. So,
$$\mathcal{W}_1\left(\mu^{(X_\ell^{(R)})},\mu\right)\le\PR(X_\ell^{(R)}\neq X)\le\varepsilon,$$
which means that if $n$ is large enough,
$$\mathcal{W}_1(\mu_{n,\text{ER}},\mu)\le19\varepsilon,$$
concluding the proof of Theorem \ref{main-thm-2}. We, therefore, have to prove this claim.\\\\
Choose $\ell\in\N$ and $C>0$ large enough, so that Lemma \ref{right-scal-ER} holds and both $\{|X_\ell|\ge \varepsilon^{-1}\}$ and $\{\mu^{(X_\ell^{(R)})}([-C,C])\le1-\varepsilon\}$ have probability $\le\varepsilon$. Observe that if $|x|\le\varepsilon^{-1}$, for any $y\in\R$,
\begin{align*}
&\dfrac{x}{\sqrt{2(d-1)}}\cdot y^2-\left(\dfrac{1}{4(d-1)}+\dfrac{1}{12}\right)\cdot y^4\le\dfrac{x^2}{2}\le\dfrac{\varepsilon^{-2}}{2}\ \ \text{and}\\
\int_\R&\exp\left(\dfrac{x}{\sqrt{2(d-1)}}\cdot y^2-\left(\dfrac{1}{4(d-1)}+\dfrac{1}{12}\right)\cdot y^4\right) \text{d}y\\&\ge e^{-\frac{1}{4(d-1)}-\frac{1}{12}}\int_0^1\exp\left(\dfrac{x}{\sqrt{2(d-1)}}\cdot y^2\right)\text{d}y\\&\ge\exp\left(-\dfrac{1}{4(d-1)}-\dfrac{1}{12}-\dfrac{\varepsilon^{-1}}{\sqrt{2(d-1)}}\right).
\end{align*}
This means that there exists a constant $D=D(\varepsilon)>0$ such that for any $y_1<y_2$,
$$\dfrac{\int_{y_1}^{y_2}\exp\left(\dfrac{x}{\sqrt{2(d-1)}}\cdot u^2-\left(\dfrac{1}{4(d-1)}+\dfrac{1}{12}\right)\cdot u^4\right) \text{d}u}{\int_\R\exp\left(\dfrac{x}{\sqrt{2(d-1)}}\cdot u^2-\left(\dfrac{1}{4(d-1)}+\dfrac{1}{12}\right)\cdot u^4\right) \text{d}u}\le D(y_2-y_1),$$
therefore there exist $-C=y_0<y_1<\cdots<y_s=C$ such that for any $i\in\{0,1,\dots,s-1\}$,
$$\PR\left(\mu^{(X_\ell^{(R)})}([y_i,y_{i+1}])>\varepsilon\ |\ X_\ell^{(R)}\in[-\varepsilon^{-1},\varepsilon^{-1}]\right)=0.$$
Choose $\eta_1=\sqrt{\frac{d-1}{d}}\cdot\varepsilon,\eta_2,\eta>0$ such that: 
\begin{align*}
\PR(\exp(\tilde{W}_k)\le\eta_1)\le\varepsilon,\ \ \eta_2\le\dfrac{\tilde{c}_0}{4}\sqrt{\dfrac{d-1}{d}}\varepsilon^2\exp\left(-\dfrac{\varepsilon^{-1}}{\sqrt{2(d-1)}}-\dfrac{1}{4(d-1)}\right),\\  \frac{2\eta_2}{\eta_1\tilde{c}_1-\eta_2}\le\varepsilon\ \ \text{and}\ \ \frac{s\eta}{\eta_2^2}\le\varepsilon,
\end{align*}
where
$$\tilde{c}_1:=\exp\left(-\dfrac{C^2\varepsilon^{-1}}{\sqrt{2(d-1)}}-\dfrac{C^4}{4(d-1)}\right)>0.$$
This choice is possible due to an argument almost identical to that of (\ref{Wk-size}). Also, for each $i\in\{1,2,\dots,s\}$, let
$$\tilde{Z}^{(i)}=\SL{-Cn^{3/4}\le m\le y_i n^{3/4}}\tilde{Z}_m.$$
Because of Proposition \ref{Zm-ER-mom}, we can set
$$\sup\limits_n\dfrac{\sup\limits_{|m|\le Cn^{3/4}}\E(\tilde{Z}_m)}{\inf\limits_{|m|\le Cn^{3/4}}\E(\tilde{Z}_m)}=:\tilde{M}<\infty.$$
For any $i$, let $\tilde{N}_i=\#\{m:-Cn^{3/4}\le m\le y_in^{3/4}\}$. Then,
\begin{align*}
&\E\left[\left(\tilde{Z}^{(i)}-\exp\left(\tilde{W}_{k,n}^{(R)}\right)\SL{-Cn^{3/4}\le m\le y_in^{3/4}}\E(\tilde{Z}_m)\cdot\exp\left(\theta_m\widehat{X}_{\ell,n}^{(R)}-\theta_m^2/2\right)\right)^2\right]
\\=\ &\E\left[\left(\SL{-Cn^{3/4}\le m\le y_in^{3/4}}\left[\tilde{Z}_m-\exp\left(W_{k,n}^{(R)}+\theta_m\widehat{X}_{\ell,n}^{(R)}-\theta_m^2/2\right)\E(\tilde{Z}_m)\right]\right)^2\right]
\\\le\ &\tilde{N}_i^2\cdot2\tilde{\alpha}_{\ell,k,R,C}\cdot\left(\sup\limits_{-Cn^{3/4}\le m\le y_in^{3/4}}\E(\tilde{Z}_m)\right)^2\\\le\ &2\tilde{M}^2\tilde{\alpha}_{\ell,k,R,C}\cdot(\E(\tilde{Z}^{(i)}))^2.
\end{align*}
Therefore, for any $1\le i\le s$, if $k,\ell,R$ and $n$ are large enough,
\begin{equation}\label{Zi-appr-ER}
\E\left[\left(\tilde{Z}^{(i)}-\exp\left(\tilde{W}_{k,n}^{(R)}\right)\SL{-Cn^{3/4}\le m\le y_in^{3/4}}\E(\tilde{Z}_m)\cdot\exp\left(\theta_m\widehat{X}_{\ell,n}^{(R)}-\theta_m^2/2\right)\right)^2\right]\le\ \eta\cdot\E(\tilde{Z}^{(i)})^2.
\end{equation}
Consider the event
\begin{align*}
B_n^{(1)}:=&\ \left\{|\widehat{X}_{\ell,n}^{(R)}|\le \varepsilon^{-1}\right\}\cap\{\mu^{(X_\ell^{(R)})}([-C,C])>1-\varepsilon\}
\cap\left\{\exp\left(\tilde{W}_{k,n}^{(R)}\right)>\eta_1\right\}\\&\cap\bigcap\limits_{i=1}^s\left\{\dfrac{\left|\tilde{Z}^{(i)}-\exp\left(\tilde{W}_{k,n}^{(R)}\right)\SL{-Cn^{3/4}\le m\le y_in^{3/4}}\E(\tilde{Z}_m)\exp(\theta_m\widehat{X}_{\ell,n}^{(R)}-\theta_m^2/2)\right|}{\E(\tilde{Z}^{(i)})}\le\eta_2\right\}\\
&\cap\left\{\sup\limits_{a,b\in[-C,C]}\left| \dfrac{\SL{an^{3/4}\le m\le bn^{3/4}}\E(\tilde{Z}_m)\exp(\theta_m\widehat{X}_{\ell,n}^{(R)}-\theta_m^2/2)}{\tilde{c}_n\int_a^b\exp\left(\frac{y^2}{\sqrt{2(d-1)}}X_\ell^{(R)}-\left(\frac{1}{12}+\frac{1}{4(d-1)}\right)y^4\right)\ \text{d}y}-1\right|\le\frac{\varepsilon}{2} \right\},
\end{align*}
where $\tilde{c}_n=\frac{\sqrt[4]{n}}{\sqrt{2\pi}}\cdot\exp(-\frac34)$. On the event $B_n^{(1)}$,
\begin{align*}
\tilde{Z}^{(s)}&\ge\exp\left(\tilde{W}_{k,n}^{(R)}\right)\SL{|m|\le Cn^{3/4}}\left[\E(\tilde{Z}_m)\exp(\theta_m\widehat{X}_{\ell,n}^{(R)}-\theta_m^2/2)\right]
-\eta_2\cdot\E(\tilde{Z}^{(s)})
\\&\ge\eta_1\exp\left( -\dfrac{\varepsilon^{-1}}{\sqrt{2(d-1)}}-\dfrac{1}{4(d-1)}\right)\SL{|m|\le n^{3/4}}\E(\tilde{Z}_m)-\eta_2\cdot\E(\tilde{Z}^{(s)})
\\&\ge\left[\eta_1\exp\left(-\dfrac{\varepsilon^{-1}}{\sqrt{2(d-1)}}-\dfrac{1}{4(d-1)}\right)\dfrac{\tilde{c}_0}{2}-\eta_2\right]\cdot\E(\tilde{Z}^{(s)})
\\&\ge\dfrac{\tilde{c}_0}{4}\sqrt{\dfrac{d-1}{d}}\cdot\varepsilon\exp\left(-\dfrac{\varepsilon^{-1}}{\sqrt{2(d-1)}}-\dfrac{1}{4(d-1)}\right)\cdot\E(\tilde{Z}^{(s)}),
\end{align*}
due to the way $\tilde{c}_0,\eta_1$ and $\eta_2$ were defined. Therefore, on the event
$$B_n:=B_n^{(1)}\cap\left\{\SL{|m|>Cn^{3/4}}\tilde{Z}_m\le\dfrac{\tilde{c}_0}{4}\sqrt{\dfrac{d}{d-1}}\varepsilon^2\exp\left(-\dfrac{\varepsilon^{-1}}{\sqrt{2(d-1)}}-\dfrac{1}{4(d-1)}\right)\cdot\E(\tilde{Z}^{(s)})\right\}$$
we will have that $\frac{\tilde{Z}-\tilde{Z}^{(s)}}{\tilde{Z}}<\varepsilon$. We observe that on the event $B_n^{(1)}$, since $-C\le y_i\le C$,
$$\left|\dfrac{\SL{-Cn^{3/4}\le m\le y_in^{3/4}}\E(\tilde{Z}_m)\exp(\theta_m\widehat{X}_{\ell,n}^{(R)}-\theta_m^2/2)}{\SL{|m|\le Cn^{3/4}}\E(\tilde{Z}_m)\exp(\theta_m\widehat{X}_{\ell,n}^{(R)}-\theta_m^2/2)}-\dfrac{\mu^{(X_\ell^{(R)})}([-C,y_i])}{\mu^{(X_\ell^{(R)})}([-C,C])}\right|\le\varepsilon.$$
We claim that for large $n$, $\PR(B_n)\ge1-8\varepsilon$. Due to (\ref{Zi-appr-ER}) and Chebyshev's inequality,
\begin{align*}
\PR\left(\bigcup\limits_{i=1}^s\left\{\dfrac{\left|\tilde{Z}^{(i)}-\exp\left(\tilde{W}_{k,n}^{(R)}\right)\SL{-Cn^{3/4}\le m\le y_in^{3/4}}\E(\tilde{Z}_m)\exp(\theta_m\widehat{X}_{\ell,n}^{(R)}-\theta_m^2/2)\right|}{\E(\tilde{Z}^{(i)})}>\eta_2\right\}\right)&\le s\cdot\dfrac{\eta}{\eta_2^2}\\&\le\varepsilon.
\end{align*}
Also, due to the way $\eta_1$ was chosen, if $n$ is large enough,
$$\PR\left(\exp\left(\tilde{W}_{k,n}^{(R)}\right)\le\eta_1\right)\le2\varepsilon\ \ \text{and}\ \ \PR\left(|\widehat{X}_{\ell,n}^{(R)}|>\varepsilon^{-1}\right)\le2\varepsilon.$$
Moreover, for any sequence $(x_n)_{n=1}^\infty$ with $|x_n|\le\varepsilon^{-1}$ which converges to $x$, uniformly for $a,b\in[-C,C]$ it holds that
$$\SL{an^{3/4}\le m\le bn^{3/4}}\E(\tilde{Z}_m)\exp(\theta_mx_n-\theta_m^2/2)\sim\tilde{c}_{n}\int_a^b\exp\left(\tfrac{y^2}{\sqrt{2(d-1)}}x-(\tfrac{1}{4(d-1)}+\tfrac{1}{12})y^4\right)\text{d}y.$$
Applying this for the sequence $\widehat{X}_{\ell,n}^{(R)}\xrightarrow[n\to\infty]{}X_\ell^{(R)}$ we get that $\PR(B_n^{(1)})\ge1-7\varepsilon$ for large enough $n$ and therefore Lemma \ref{right-scal-ER} implies that $\PR(B_n)\ge1-8\varepsilon$.\\\\
We will prove that on the event $B_n$,
\begin{equation}\label{conv-y_i-ER}
\sup\limits_{1\le i\le s}|\mu_{n,\text{ER}}([-C,y_i])-\mu^{(X_\ell^{(R)})}([-C,y_i])|<4\varepsilon.
\end{equation}
With an argument identical to that included in the proof of Theorem \ref{main-thm-1}, one can see that this is enough to finish the proof.
Indeed,
\begin{align*}
&|\mu_{n,\text{ER}}([-C,y_i])-\mu^{(X_\ell^{(R)})}([-C,y_i])|\\=\ & \left|\dfrac{\tilde{Z}^{(i)}}{\tilde{Z}}-\mu^{(X_\ell^{(R)})}([-C,y_i])\right|
\\\le\ & \left|\dfrac{\tilde{Z}^{(i)}}{\tilde{Z}}-\dfrac{\tilde{Z}^{(i)}}{\tilde{Z}^{(s)}}\right|+\left|\mu^{(X_\ell^{(R)})}([-C,y_i])-\dfrac{\mu^{(X_\ell^{(R)})}([-C,y_i])}{\mu^{(X_\ell^{(R)})}([-C,C])}\right|\\+&\left|\dfrac{\tilde{Z}^{(i)}}{\tilde{Z}^{(s)}}-\dfrac{\SL{-Cn^{3/4}\le m\le y_in^{3/4}}\E(\tilde{Z}_m)\exp(\theta_m\widehat{X}_{\ell,n}^{(R)}-\theta_m^2/2)}{\SL{-Cn^{3/4}\le m\le Cn^{3/4}}\E(\tilde{Z}_m)\exp(\theta_m\widehat{X}_{\ell,n}^{(R)}-\theta_m^2/2)}\right|\\+&\left|\dfrac{\SL{-Cn^{3/4}\le m\le y_in^{3/4}}\E(\tilde{Z}_m)\exp(\theta_m\widehat{X}_{\ell,n}^{(R)}-\theta_m^2/2)}{\SL{-Cn^{3/4}\le m\le Cn^{3/4}}\E(\tilde{Z}_m)\exp(\theta_m\widehat{X}_{\ell,n}^{(R)}-\theta_m^2/2)}-\dfrac{\mu^{(X_\ell^{(R)})}([-C,y_i])}{\mu^{(X_\ell^{(R)})}([-C,C])} \right|.
\end{align*}
We show that on the event $B_n$, each one of these terms is $<\varepsilon$.
\begin{enumerate}
\item 
On $B_n$, $\frac{\tilde{Z}-\tilde{Z}^{(s)}}{\tilde{Z}}<\varepsilon$, therefore the first two terms are $<\varepsilon$.
\item 
As we already explained, on $B_n$, the last term is also $<\varepsilon$.
\item 
On the event $B_n$, $\tilde{Z}^{(s)}\ge(\eta_1c_1-\eta_2)\cdot\E(\tilde{Z}^{(s)})$ and 
$$\left|\dfrac{\tilde{Z}^{(i)}}{\tilde{Z}^{(s)}}-\dfrac{\SL{-Cn^{3/4}\le m\le y_in^{3/4}}\E(\tilde{Z}_m)\exp(\theta_m\widehat{X}_{\ell,n}^{(R)}-\theta_m^2/2)}{\SL{-Cn^{3/4}\le m\le Cn^{3/4}}\E(\tilde{Z}_m)\exp(\theta_m\widehat{X}_{\ell,n}^{(R)}-\theta_m^2/2)}\right|\le\dfrac{2\eta_2}{\eta_1c_1-\eta_2}\le\varepsilon.$$
\end{enumerate}
The proof is complete.
\end{proof}
It remains to prove Proposition \ref{Zm-appr-ER}.
\begin{proof}[Proof of Proposition \ref{Zm-appr-ER}]
We prove the desired statement, with
\begin{align*}
\tilde{\alpha}_{\ell,k,R,C}:=\sup\limits_{|x|\le C}\Biggl\{&\exp\left(\dfrac{x^4}{2(d-1)}+\SL{i=3}^\infty\dfrac{d^{-i}}{2i}\right)\Biggr. \\ \Biggl. &-2\E\left[\exp\left(\frac{x^2}{\sqrt{2(d-1)}} X_\ell'^{(R)}-\frac{x^4}{4(d-1)}+\tilde{W}_k'^{(R)}\right)\right]\Biggr.\\&+\E\left[\exp\left(\frac{\sqrt{2}\cdot x^2}{\sqrt{d-1}}X_\ell^{(R)}-\frac{x^4}{2(d-1)}+2\tilde{W}_k^{(R)}\right)\right]\Bigg\}.
\end{align*}
This $\tilde{\alpha}_{\ell,k,R,C}$ satisfies the desired property. Indeed, for any $x\in[-C,C]$,
\begin{align*}
&\left|\exp\left(\dfrac{x^4}{2(d-1)}+\SL{i=3}^\infty\dfrac{d^{-i}}{2i}\right)-\E\left[\exp\left(\frac{x^2}{\sqrt{2(d-1)}} X_\ell'^{(R)}-\frac{x^4}{4(d-1)}+\tilde{W}_k'^{(R)}\right)\right]\right|
\\\le&\exp\left(\dfrac{x^4}{2(d-1)}\right)\left|\exp\left(\SL{i=3}^\infty\dfrac{d^{-i}}{2i}\right)-\E\left[\exp\left(\tilde{W}_k'^{(R)}\right)\right]\right|
\\&+\E\left[\exp\left(\tilde{W}_k'^{(R)}\right)\right]\left|\exp\left(\dfrac{x^4}{2(d-1)}\right)-\E\left[\exp\left(\frac{x^2}{\sqrt{2(d-1)}} X_\ell'^{(R)}-\frac{x^4}{4(d-1)}\right)\right]\right|.
\end{align*}
As explained in the proof of Proposition \ref{Zm-appr-reg}, 
$$\lim\limits_{R\to\infty}\lim\limits_{k\to\infty}\E\left[\exp\left(\tilde{W}_k'^{(R)}\right)\right]=\exp\left(\SL{i=3}^\infty\dfrac{d^{-i}}{2i}\right).$$
Also, since
\begin{align*}
&\ \left|\exp\left(\dfrac{x^4}{2(d-1)}\right)-\E\left[\exp\left(\frac{x^2}{\sqrt{2(d-1)}} X_\ell'^{(R)}-\frac{x^4}{4(d-1)}\right)\right]\right|
\\\le&\ \exp\left(\dfrac{x^4}{2(d-1)}\right)\left|\exp\left(\dfrac{x^4}{4(d-1)}(2\gamma_\ell^{(1)}+\gamma_\ell^{(2)})\right)-1\right|\\&+\E\left|\exp\left(\dfrac{x^2}{\sqrt{2(d-1)}}X_\ell'\right)-\exp\left(\dfrac{x^2}{\sqrt{2(d-1)}}X_\ell'^{(R)}\right)\right|
\\\le&\ \exp\left(\dfrac{C^4}{2(d-1)}\right)\cdot\dfrac{C^4}{d-1}\cdot(2\gamma_\ell^{(1)}+\gamma_\ell^{(2)})+\E\left[\mathbf{1}_{|X_\ell'|>R}\left(\exp\left(\dfrac{x^2}{\sqrt{2(d-1)}}X_\ell'\right)+1\right)\right]
\end{align*}
and
\begin{align*}
\E\left[\mathbf{1}_{|X_\ell'|>R}\left(\exp\left(\dfrac{x^2}{\sqrt{2(d-1)}}X_\ell'\right)+1\right)\right]&\le\sqrt{2\PR(|X_\ell'|>R)\cdot\E\left[\exp\left(\dfrac{\sqrt{2}\cdot x^2}{\sqrt{d-1}}X_\ell'\right)+1\right]}
\\&\le\sqrt{4\cdot\PR(|X_\ell'|>R)\cdot\exp\left(\dfrac{4C^4}{d-1}\right)},
\end{align*}
which hold if $\ell$ is large enough, it follows that
\begin{align*}
\lim\limits_{R\to\infty}\lim\limits_{\ell\to\infty}\lim\limits_{k\to\infty}\sup\limits_{|x|\le  C}&\left| \exp\left(\dfrac{x^4}{2(d-1)}+\SL{i=3}^\infty\dfrac{d^{-i}}{2i}\right)\right.\\&-\left. \E\left[\exp\left(\frac{x^2}{\sqrt{2(d-1)}} X_\ell'^{(R)}-\frac{x^4}{4(d-1)}\right)\right]\E\left[\exp\left(\tilde{W}_k'^{(R)}\right)\right]\right|=0.
\end{align*}
In a similar fashion, we can prove that
\begin{align*}
\lim\limits_{R\to\infty}\lim\limits_{\ell\to\infty}\lim\limits_{k\to\infty}\sup\limits_{|x|\le  C}&\left| \exp\left(\dfrac{x^4}{2(d-1)}+\SL{i=3}^\infty\dfrac{d^{-i}}{2i}\right)\right.\\&-\left. \E\left[\exp\left(\frac{\sqrt{2}\cdot x^2}{\sqrt{d-1}} X_\ell^{(R)}-\frac{x^4}{2(d-1)}+2\tilde{W}_k^{(R)}\right)\right]\right|=0.
\end{align*}
This implies that indeed, $\lim\limits_{R\to\infty}\lim\limits_{\ell\to\infty}\lim\limits_{k\to\infty}\tilde{\alpha}_{\ell,k,R,C}=0$.\\\\ 
Suppose Proposition \ref{Zm-appr-ER} is false. Then, there exists some sequence $m=m(n)$ such that $|m|\le Cn^{3/4}$ for any $n$ and
$$\limsup\limits_{n\to\infty}\E\left[\left(\dfrac{\tilde{Z}_m}{\E(\tilde{Z}_m)}-\exp\left(\theta_m\widehat{X}_{\ell,n}^{(R)}-\theta_m^2/2+\tilde{W}_{k,n}^{(R)}\right)\right)^2\right]>\tilde{\alpha}_{\ell,k,R,C}.$$
Without loss of generality, we may assume (by extracting a subsequence and, possibly, modifying some terms) that $m\cdot n^{-3/4}\xrightarrow[n\to\infty]{}x\in[-C,C]$. Then, due to Lemma \ref{Zm-ER-mom} and Proposition \ref{cyc-path-ER},
\begin{align*}
&\E\left[\left(\dfrac{\tilde{Z}_m}{\E(\tilde{Z}_m)}-\exp\left(\theta_m\widehat{X}_{\ell,n}^{(R)}-\theta_m^2/2+\tilde{W}_{k,n}^{(R)}\right)\right)^2\right]
\\ =\ &\dfrac{\E(\tilde{Z}_m^2)}{\E(Z_m)^2}-2\cdot\E_m^*\left[\exp\left(\theta_m\widehat{X}_{\ell,n}^{(R)}-\theta_m^2/2+\tilde{W}_{k,n}^{(R)}\right)\right]+\E\left[\exp\left(2\theta_m\widehat{X}_{\ell,n}^{(R)}-\theta_m^2+2\tilde{W}_{k,n}^{(R)}\right)\right]
\\\underset{n\to\infty}{\to}&\exp\left(\frac{x^4}{2(d-1)}+\SL{i=3}^\infty\frac{d^{-i}}{2i}\right)-2\E\left[\exp\left(\frac{x^2}{\sqrt{2(d-1)}} X_\ell'^{(R)}-\frac{x^4}{4(d-1)}+\tilde{W'}_k^{(R)}\right)\right]
\\&+\E\left[\exp\left(\frac{\sqrt{2}\cdot x^2}{\sqrt{d-1}}X_\ell^{(R)}-\frac{x^4}{2(d-1)}+2\tilde{W}_k^{(R)}\right)\right].
\end{align*}
Due to the definition of $\tilde{\alpha}_{\ell,k,R,C}$ that we gave above, we have reached a contradiction and the proof is complete.
\end{proof}
\subsection{Proofs of Corollaries}\label{3.3}
In this subsection, we prove Corollaries \ref{Glaub-lower} and \ref{free-evergy-ER}.
\begin{proof}[Proof of Corollary \ref{Glaub-lower}]
As is well-known, the spectral gap of a Markov chain $(X_t)_{t\ge0}$, denoted by $\text{gap}(X_t)$, is defined to be
$$\text{gap}(X_t):=\inf\limits_{f\neq c}\dfrac{\mathcal{E}(f,f)}{\text{Var}_\pi(f)},$$
where the $\inf$ is over all non-constant functions $f:\mathcal{X}\to\R$ and $\mathcal{E}(f,f)$ is the Dirichlet form of $f$, i.e.
$$\mathcal{E}(f,f):=\dfrac{1}{2}\SL{x,y\in\mathcal{X}}\pi(x)q(x,y)(f(x)-f(y))^2.$$
We set $f(\sigma)=m(\sigma)=\SL{v}\sigma_v$. Then,
$$\mathcal{E}(m,m)\le2\cdot\SL{\sigma,\tau}\pi(\sigma)q(\sigma,\tau)\le2n.$$
In the regular graph case, due to Theorem \ref{main-thm-1},
$$\mu_{n,\text{reg}}(\R\setminus(-1,1))\xrightarrow[n\to\infty]{\PR}\mu_{\text{reg}}(\R\setminus(-1,1))>0.$$
If $F_n=\{G:\PR(|m|\ge n^{3/4}|G)>\mu_{\text{reg}}(\R\setminus(-1,1))/2\}$, then $\PR_{G\sim G_{n,d}}(F_n^c)\xrightarrow[n\to\infty]{}0$ and on the event $F_n$,
$$\text{Var}(m)\ge n^{3/2}\cdot\PR(|m|\ge n^{3/4})\ge\frac12\mu_{\text{reg}}(\R\setminus(-1,1))\cdot n^{3/2}.$$
This implies that
$$\text{gap}(X_t)\le\dfrac{\mathcal{E}(m,m)}{\text{Var}(m)}\le cn^{-1/2}.$$
Due to the well-known connection between the spectral gap and the mixing time, it follows that
$$t_{\text{mix}}(1/4)\ge c_1\cdot\text{gap}^{-1}(X_t)\ge c_0\sqrt{n}$$
and the first statement follows.\\\\
As for the Erd\"os-R\'enyi case, let $\varepsilon>0$. Let $K>0$ be such that $\PR(|X|\le K)\ge1-\frac{\varepsilon}{2}$. Also, for some constant $c(K)$, $\mu^{(x)}(\R\setminus(-1,1))\ge c(K)$ for any $|x|\le K$. Finally, if $n$ is large enough, there exists a coupling of $\mu_n$ and $\mu$ such that $\E(\text{d}_{\text{KS}}(\mu_n,\mu))\le\frac{\varepsilon c(K)}{8}$. Then,
$$\PR(\text{d}_{\text{KS}}(\mu_n,\mu)>c(K)/4)\le\varepsilon/2\ \ \Rightarrow\ \ \PR\left[|\mu_n(\R\setminus(-1,1))-\mu(\R\setminus(-1,1))|>c(K)/2\right]\le\varepsilon/2.$$
Therefore, with probability $\ge1-\varepsilon$, we know that $\mu_n(\R\setminus(-1,1))\ge c(K)/2$, and, with an argument similar to the one for the regular case, on this event it is true that
$$\text{Var}(m)\ge\delta_1n^{3/2},$$
which then implies the second statement of the Corollary.
\end{proof}
We move on to Corollary \ref{free-evergy-ER}.
\begin{proof}[Proof of Corollary \ref{free-evergy-ER}]
We revisit the proof of Theorem \ref{main-thm-2}. Let $E_n$ be the number of edges of the graph. Suppose all the relevant random variables are in the same probability space. Observe that $Z_n=2^n\cosh(\beta)^{|E_n|}\cdot\tilde{Z}$ and if we set
\begin{align*}
\Delta_n':=\Delta_n-&\SL{i=3}^\infty\left(\tilde{C}_i\log(1+d^{-i})-\frac{d^i+1}{2i}\right)\\-&\log\int_\R\exp\left(\frac{y^2}{\sqrt{2(d-1)}}X-\left(\dfrac{1}{12}+\dfrac{1}{4(d-1)}\right)y^4\right)\ \text{d}y,
\end{align*}
then
\begin{align*}
\{|\Delta_n'|\ge6\varepsilon\}\subseteq&\ \{|\log\tilde{Z}-\log\tilde{Z}^{(s)}|\ge\varepsilon\}
\\\cup &\ \left\{\left|\log\tilde{Z}^{(s)}-\tilde{W}_{k,n}^{(R)}-\log\SL{|m|\le Cn^{3/4}}\E(\tilde{Z}_m)\exp(\theta_m\widehat{X}_{\ell,n}^{(R)}-\theta_m^2/2)\right|\ge\varepsilon\right\}
\\\cup &\ \left\{|\tilde{W}_{k,n}^{(R)}-\tilde{W}_k^{(R)}|\ge\varepsilon\right\}\ \cup\left\{|\tilde{W}_k^{(R)}-\tilde{W}_\infty|\ge\varepsilon\right\}
\\\cup &\ \left\{\left|\log\SL{|m|\le Cn^{3/4}}\E(\tilde{Z}_m)\exp(\theta_m\widehat{X}_{\ell,n}^{(R)}-\theta_m^2/2)-\log\tilde{c}_n\right.\right.\\&\left.\left.-\log\int_{-C}^C\exp\left(\frac{y^2}{\sqrt{2(d-1)}}X_\ell^{(R)}-\left(\frac{1}{12}+\frac{1}{4(d-1)}\right)y^4\right)\ \text{d}y\right|\ge\varepsilon\right\}
\\\cup &\ \left\{\left|\log\mu^{(X_\ell^{(R)})}([-C,C])\right|\ge\varepsilon\right\}
\\\cup &\ \{X_\ell^{(R)}\neq X\}.
\end{align*}
This means that
$$\{|\Delta_n'|\ge6\varepsilon\}\subseteq B_n\cup\left\{|\tilde{W}_{k,n}^{(R)}-\tilde{W}_k^{(R)}|\ge\varepsilon\right\}\ \cup\left\{|\tilde{W}_k^{(R)}-\tilde{W}_\infty|\ge\varepsilon\right\}\cup\{X_\ell^{(R)}\neq X\},$$
therefore, as we have already explained, if $k,R$ and $n$ are large enough,
$$\PR(|\Delta_n'|\ge6\varepsilon)\le11\varepsilon.$$
This shows that indeed, $\Delta_n'\xrightarrow[n\to\infty]{\PR}0$.
\end{proof}
\section{Moment calculations}\label{4}
Before explaining the calculations, we state and prove the following technical lemma, which will be useful throughout this section.
\begin{lem}\label{comp-fn-Taylor}
Let $D\subseteq\R^d$ be a compact set and $f_1,f_2:D\to\R$ two continuous functions with the following properties:
\begin{itemize}
\item 
$f_1$ has a unique global maximum at $x\in D$.
\item 
$f_2$ has a unique global minimum at $x$, with $f_2(x)=0$.
\item 
$c_0:=\frac12\liminf\limits_{\substack{y\to x\\ y\in D}}\frac{f_1(x)-f_1(y)}{f_2(y)}>0$.
\end{itemize}
Then, for some $c>0$, $f_1(y)\le f_1(x)-cf_2(y)$ for any $y\in D$.
\end{lem}
\begin{proof}
Due to the third property, there exists some $\delta>0$ such that
$$\dfrac{f_1(x)-f_1(y)}{f_2(y)}\ge c_0\ \Leftrightarrow f_1(y)\le f_1(x)-c_0f_2(y),$$
for all $y$ such that $|y-x|\le\delta$. On the other hand, because of the first two properties, the continuity of the functions and the compactness of $D$, there exists some $\varepsilon>0$ such that $f_1(y)\le f_1(x)-\varepsilon$ and $f_2(y)\le\varepsilon^{-1}$ for any $y$ such that $|y-x|\ge\delta$. Therefore, for any such $y$,
$$f_1(y)\le f_1(x)-\varepsilon\le f_1(x)-\varepsilon^2f_2(y).$$
The Lemma follows, if we set $c=\min(c_0,\varepsilon^2)$.
\end{proof}
\begin{rem}
We will use this Lemma in situations in which $f_1$ is $C^\infty$ around its maximizer and $f_2$ is part of the Taylor expansion of $f_1$ around the maximizer.
\end{rem}
\subsection{Regular Graphs}
The goal of this subsection is to prove Proposition \ref{Zm-reg-mom}. For the second moment calculation, we will need the following Proposition, regarding the maximizer of $\Phi$, when the vertex-empirical distribution $\bar{h}$ is fixed.
\begin{prop}\label{max-given-vtx-emp}
Let $\mu$ be a spin system on a $G\sim G_{n,d}$, with a finite spin set $\mathcal{X}$ and weights $(\overline{\psi},\psi)$. Assume that $\psi>0$. For a given vertex-empirical distribution $\bar{h}$ such that $\bar{h}(x)>0$ for any $x\in\mathcal{X}$, let 
$$B_{\bar{h}}=\left\{h\in[0,1]^{\R^{\mathcal{X}^2}}:h(x,y)=h(y,x)\ \forall x,y\ \ \text{and}\ \ \SL{y\in\mathcal{X}}h(x,y)=\bar{h}(x)\ \forall x\right\}$$ 
and $\Phi_{\bar{h}}:B_{\bar{h}}\to\R$ be the function
$$\Phi_{\bar{h}}(h):=H(h)+\langle h,\log\psi\rangle.$$
Then, the maximizer of $\Phi_{\bar{h}}$ is an edge-empirical distribution $h^*$ of the form
\begin{equation}\label{h*-def}
h^*(x,y)\varpropto q(x)q(y)\psi(x,y),
\end{equation}
for some $q:\mathcal{X}\to(0,\infty)$ such that $h^*\in B_{\overline{h}}$. 
\end{prop}
\begin{proof}
Without loss of generality, we may assume that $\SL{x,y\in\mathcal{X}}\psi(x,y)=1$. At first, we prove that a $q$ such that the $h^*$ defined above is in $B_{\overline{h}}$ exists. For that purpose, we use Brouwer's fixed point Theorem. Without loss of generality, we will look for $q\in\mathcal{M}_1(\mathcal{X})$. Define the function $F:\mathcal{M}_1(\mathcal{X})\to\mathcal{M}_1(\mathcal{X})$ to be
$$F(q)(x)=\dfrac{\frac{\overline{h}(x)}{\SL{y\in\mathcal{X}}\psi(x,y)q(y)}}{\SL{z\in\mathcal{X}}\frac{\overline{h}(z)}{\SL{y\in\mathcal{X}}\psi(z,y)q(y)}}.$$
Then, because $\mathcal{X}$ is finite and $\psi>0$, the denominators stay uniformly away from 0 and, therefore, $F$ is continuous. So, due to Brouwer's fixed point Theorem, $F$ has a fixed point $q$, which satisfies the conditions needed.\\\\
Let $q$ satisfy the conditions discussed above and write $h^*(x,y)=q(x)q(y)\psi(x,y)/z$ for some $z>0$. Observe that
$$H(h)+\langle h,\log\psi\rangle=-D_{\text{KL}}(h||\psi).$$
Therefore, for any $h\in B_{\overline{h}}$,
\begin{align*}
D_{\text{KL}}(h||\psi)-D_{\text{KL}}(h||h^*)&=\SL{x,y\in\mathcal{X}}h(x,y)\log\dfrac{h^*(x,y)}{\psi(x,y)}
\\&=-\log(z)+\SL{x,y\in\mathcal{X}}h(x,y)(\log q(x)+\log q(y))
\\&=-\log(z)+2\cdot\SL{x\in\mathcal{X}}\overline{h}(x)\log q(x).
\end{align*}
Observe that this expression does not depend on $h\in B_{\bar{h}}$, so it turns out that
$$D_{\text{KL}}(h||\psi)=D_{\text{KL}}(h||h^*)+D_{\text{KL}}(h^*||\psi)\ge D_{\text{KL}}(h^*||\psi),$$
proving that indeed, $h^*$ is the maximizer of $\Phi$. Also, since $D_{\text{KL}}(\mu||\nu)=0\Leftrightarrow\mu=\nu$, this maximizer is unique.
\end{proof}
\subsubsection{First moment calculation}
We perform the calculations proving Lemma \ref{right-scal-reg} and the first statement of Lemma \ref{Zm-reg-mom}. To that end, we use (\ref{h-contr-reg}) for the Ising model. At the critical temperature, we are still in the uniqueness regime, so the unique maximizer of $\Phi_1$ is at $h_*(x,x')\varpropto\exp(\beta xx')$. This means that $h_*(1,1)=h_*(-1,-1)=\frac{d}{4(d-1)}$ and $h_*(1,-1)=h_*(-1,1)=\frac{d-2}{4(d-1)}$. Set $$m(h)=n(h(1,1)-h(-1,-1))\ \ \text{and}\ \ s(h)=n(h(1,1)+h(-1,-1)).$$ For brevity purposes, whenever $h$ is implied, we will just use $s$ instead of $s(h)$. Let $\Phi_1$ be the function $\Phi$ corresponding to the Ising model and
$$\mathcal{H}_{n,m}=\{h\in\mathcal{H}_n:m(h)=m\},\ \mathcal{H}_{n,m}^{(1)}=\left\{h\in\mathcal{H}_{n,m}: \left|s(h)-\frac{d}{2(d-1)}\cdot n\right|\le n^{3/5}\right\}.$$
Around $h_*$, for $h\in\mathcal{H}_{n,m}$,
\begin{align*}
\Phi_1(h)=\Phi_1(h_*)&-\dfrac{(d-1)^2}{d-2}\cdot\left(\dfrac{s}{n}-\dfrac{d}{2(d-1)}\right)^2-\dfrac{(d-1)(d-2)(3d-2)}{12d^2}\cdot\left(\dfrac{m}{n}\right)^4\\&+\dfrac{(d-1)^2}{d}\cdot\left(\dfrac{s}{n}-\dfrac{d}{2(d-1)}\right)\cdot\left(\dfrac{m}{n}\right)^2+O\left(\left|\dfrac{s}{n}-\dfrac{d}{2(d-1)}\right|^3+\left(\dfrac{m}{n}\right)^6\right).
\end{align*}
Due to Lemma \ref{comp-fn-Taylor}, there exists some constant $c>0$ such that 
\begin{equation}\label{decay-Phi1}
\Phi_1(h)\le\Phi_1(h_*)-c\left(\left(\dfrac{s}{n}-\dfrac{d}{2(d-1)}\right)^2+\left(\dfrac{m}{n}\right)^4\right)
\end{equation}
for every $h$.
\begin{proof}[Proof of Lemma \ref{right-scal-reg}]
Let $\mathcal{H}_n^{(1)}$ be the set of $h$ such that $(\frac{s}{n}-\frac{d}{2(d-1)})^2+(\frac{m}{n})^4\ge n^{-4/5}$. For every such $h$ we use the bound in (\ref{decay-Phi1}) to show that
$$\E(Z_h)\le\exp(-cn^{1/5})\cdot n^{\Theta(1)}\exp(n\Phi_1(h_*)),$$
which implies that
\begin{equation}\label{contr-H1}
\SL{h\in\mathcal{H}_n^{(1)}}\E(Z_h)\le\exp(-cn^{1/5})\cdot n^{\Theta(1)}\exp(n\Phi_1(h_*)).
\end{equation}
Also, let $\mathcal{H}_n^{(2)}$ be the set of $h$ such that $n^{-4/5}>(\frac{s}{n}-\frac{d}{2(d-1)})^2+(\frac{m}{n})^4\ge\frac{C^4}{n}$ for some (large) constant $C>0$. For every $h\in\mathcal{H}_n^{(2)}$, due to Stirling's formula,
$$\E(Z_h)\asymp n^{-1}\exp(n\Phi_1(h)),$$
therefore
\begin{align}\label{contr-H2}
\nonumber\SL{h\in\mathcal{H}_n^{(2)}}\E(Z_h)&\lesssim n^{-1}\SL{h\in\mathcal{H}_n^{(2)}}\exp(n\Phi_1(h))
\\\nonumber&\le n^{-1}\exp(n\Phi_1(h_*))\SL{h\in\mathcal{H}_n^{(2)}}\exp\left[
-cn\left(\left(\dfrac{s}{n}-\dfrac{d}{2(d-1)}\right)^2+\left(\dfrac{m}{n}\right)^4\right)\right]
\\\nonumber&\le n^{-1}\exp(n\Phi_1(h_*))\SL{k\ge C^4}\exp(-ck)\cdot|\{h:\tfrac{k}{n}\le(\tfrac{s}{n}-\tfrac{d}{2(d-1)})^2+(\tfrac{m}{n})^4<\tfrac{k+1}{n}\}|
\\&\lesssim n^{1/4}\exp(n\Phi_1(h_*))\SL{k\ge C^4}k^{3/4}\exp(-ck),
\end{align}
as $|\{h:\frac{k}{n}\le(\tfrac{s}{n}-\frac{d}{2(d-1)})^2+(\frac{m}{n})^4<\frac{k+1}{n}\}|=O(n^{5/4}\cdot k^{3/4})$ for any $k\in\N$.\\\\
Also, if $\mathcal{H}_n^{(3)}=\{h:(\tfrac{s}{n}-\frac{d}{2(d-1)})^2+(\frac{m}{n})^4\le\frac{1}{n}\}$, then $\Phi_1(h)\ge\Phi_1(h_*)-\frac{c'}{n}$ for any $h\in\mathcal{H}_n^{(3)}$, therefore
$$\SL{|m|\le Cn^{3/4}}\E(Z_m)\ge\SL{h\in\mathcal{H}_n^{(3)}}\E(Z_h)\gtrsim n^{-1}\cdot\left|\mathcal{H}_n^{(3)}\right|\cdot\exp(n\Phi_1(h_*))\gtrsim n^{1/4}\cdot\exp(n\Phi_1(h_*)).$$
This, together with (\ref{contr-H1}), (\ref{contr-H2}) and Markov's inequality, implies the lemma.
\end{proof}
\begin{proof}[Proof of first statement of Lemma \ref{Zm-reg-mom}]
For $h\in\mathcal{H}_{n,m}\setminus\mathcal{H}_{n,m}^{(1)}$, $\Phi_1(h)\le\Phi(h_*)-cn^{-4/5}$, which implies
\begin{equation}\label{contr-h-away}
\SL{h\in\mathcal{H}_{n,m}\setminus\mathcal{H}_{n,m}^{(1)}}\E(Z_h)\le\exp(-cn^{1/5})\cdot n^{\Theta(1)}\exp(n\Phi_1(h_*)).
\end{equation}
On the other hand, for $h\in\mathcal{H}_{n,m}^{(1)}$, due to Stirling's formula,
$$\E(Z_h)\sim\dfrac{2^{5/2}(d-1)^{3/2}}{\pi nd^{3/2}(d-2)^{1/2}}\cdot\exp(n\Phi_1(h)).$$
Therefore:
\begin{align*}
\SL{h\in\mathcal{H}_{n,m}^{(1)}}\E(Z_h)&\sim\dfrac{2^{5/2}(d-1)^{3/2}}{\pi nd^{3/2}(d-2)^{1/2}}\cdot\SL{h\in\mathcal{H}_{n,m}^{(1)}}\exp(n\Phi_1(h))\\&\sim\dfrac{2^{1/2}(d-1)^{3/2}}{\pi d^{1/2}(d-2)^{1/2}}\cdot\int_\R\exp(n\Phi_1(h))\ \text{d}h\\&\sim\dfrac{2^{1/2}\sqrt{d-1}}{\sqrt{n}\sqrt{\pi d}}\cdot\exp(n\Phi_1(h_*))\cdot\exp\left(-\dfrac{(d-1)(d-2)}{12d^2}\cdot\dfrac{m^4}{n^3}\right),
\end{align*}
where in the last step we used Laplace's method. Since it is easy to check that
$$\Phi_1(h_*)=\log2+\dfrac{d}{2}\log(\cosh(\beta)),$$
this finishes the first moment calculation.
\end{proof}
\subsubsection{Second moment calculation}
\begin{proof}[Proof of the second statement of Lemma \ref{Zm-reg-mom}] We need to show that
\begin{equation}\label{sec-mom-Zm-reg}
\E(Z_m^2)\sim\dfrac{2(d-1)^{3/2}}{n\pi d\sqrt{d-2}}\cdot(4\cdot\cosh(\beta)^d)^n\cdot\exp\left(-\dfrac{(d-1)(d-2)}{6d^2}\cdot\dfrac{m^4}{n^3}\right).
\end{equation}
Observe that on a graph $G$,
\begin{equation}\label{sec-mom-expan-reg}
Z_m^2=\SL{\sigma,\tau:m(\sigma)=m(\tau)=m}\exp\left(\beta\SL{u\sim v}(\sigma_u\sigma_v+\tau_u\tau_v)\right).
\end{equation}
Consider the two-copy model on $G$, i.e. the spin system on $\mathcal{X}=\{-1,1\}^2$ with $$\psi((x,x'),(y,y'))=\exp(\beta(xy+x'y')).$$ From now on, for simplicity, we will just write $(x,x',y,y')$ instead of $((x,x'),(y,y'))$, when appropriate. Then, due to (\ref{sec-mom-expan-reg}), $Z_m^2$ is the contribution to the partition function of all the edge-empirical distributions $h$ for which both magnetizations are $m$, i.e. those satisfying the equations
\begin{align}\label{relevant-barh}
\nonumber\bar{h}(1,1)+\bar{h}(1,-1)-\bar{h}(-1,1)-\bar{h}(-1,-1)&=\bar{h}(1,1)+\bar{h}(-1,1)-\bar{h}(1,-1)-\bar{h}(-1,-1)=\frac{m}{n}\\
\Leftrightarrow\ \bar{h}(1,1)-\bar{h}(-1,-1)=\frac{m}{n}\ \ &\text{and}\ \ \bar{h}(1,-1)=\bar{h}(-1,1).
\end{align}
It is important to note that the unique maximizer $h_*$ of the function $\Phi_2$ that arises in this case is again the one for which $h(x,x',y,y')\varpropto\exp(\beta(xy+x'y'))$, i.e.
\begin{align*}
h_*(1,1,1,1)=h_*(1,-1,1,-1)=h_*(-1,1,-1,1)=h_*(-1,-1,-1,-1)&=\dfrac{d^2}{16(d-1)^2}
\\ h_*(1,1,1,-1)=h_*(1,1,-1,1)=h_*(-1,-1,1,-1)=h_*(-1,-1,-1,1)&=\dfrac{d(d-2)}{16(d-1)^2}
\\ h_*(1,1,-1,-1)=h_*(1,-1,-1,1)&=\dfrac{(d-2)^2}{16(d-1)^2}.
\end{align*}
It is also important to note that the corresponding vertex-empirical distribution is the uniform one, i.e. $\bar{h}_*(1,1)=\bar{h}_*(1,-1)=\bar{h}_*(-1,1)=\bar{h}_*(-1,-1)=1/4$.
We call $\bar{\mathcal{H}}_{n,m^{\otimes2}}$ the set of vertex-empirical distributions that satisfy the equalities (\ref{relevant-barh}) and $$\bar{\mathcal{H}}_{n,m^{\otimes2}}^{(1)}=\left\{\bar{h}\in\bar{\mathcal{H}}_{n,m^{\otimes2}}:\left|t(\bar{h})-\dfrac{1}{2}\right|\le n^{-2/5}\right\},$$ 
where for a vertex-empirical distribution $\bar{h}$ we set $t(\bar{h})=\bar{h}(1,1)+\bar{h}(-1,-1)$. Then,
\begin{align*}
\E(Z_m^2)&=\SL{\bar{h}\in\bar{\mathcal{H}}_{n,m^{\otimes2}}}\SL{h\in B_{\bar{h}}}\E(Z_h)
\end{align*}
where $B_{\bar{h}}$ was defined in Proposition \ref{max-given-vtx-emp}. For any $\bar{h}\in\bar{\mathcal{H}}_{n,m^{\otimes2}}$, we denote by $\bar{h}^*$ the maximizer of the function $\Phi_{\bar{h}}$, as explained in Proposition \ref{max-given-vtx-emp}. Using Proposition \ref{max-given-vtx-emp}, a calculation of the Taylor expansion of $\Phi_2(\bar{h}^*)$ around $h_*$, verified using Mathematica yields:
\begin{align}\label{Phi2}
\Phi_2(\overline{h}^*)=\Phi_2(h_*)&-\dfrac{2(d-1)(d-2)}{d^2-2d+2}\cdot\left(t-\dfrac{1}{2}\right)^2+\dfrac{2(d-1)(d-2)}{d^2-2d+2}\cdot\left(t-\dfrac{1}{2}\right)\cdot\left(\frac{m}{n}\right)^2\nonumber\\&-\dfrac{(d-2)(d-1)(2d^2-d+1)}{3d^2(d^2-2d+2)}\cdot\left(\frac{m}{n}\right)^4+O\left(\left|t-\dfrac{1}{2}\right|^3+\left(\dfrac{m}{n}\right)^6\right).
\end{align}
For any $\bar{h}\in\bar{\mathcal{H}}_{n,m^{\otimes2}}\setminus\bar{\mathcal{H}}_{n,m^{\otimes2}}^{(1)}$, we can see that because of Lemma \ref{comp-fn-Taylor},
\begin{align*}
\SL{h\in B_{\bar{h}}}\exp(n\Phi_2(h))&\le|B_{\bar{h}}|\cdot\exp(n\Phi_2(\bar{h}^*))\le n^{\Theta(1)}\cdot\exp(n\Phi_2(\bar{h}^*))\\&\le n^{\Theta(1)}\cdot\exp(n\Phi_2(h_*))\cdot\exp(-cn^{1/5}),
\end{align*}
therefore
$$\SL{\bar{h}\in\bar{\mathcal{H}}_{n,m^{\otimes2}}\setminus\bar{\mathcal{H}}_{n,m^{\otimes2}}^{(1)}}\SL{h\in B_{\bar{h}}}\E(Z_h)\le n^{\Theta(1)}\cdot\exp(n\Phi_2(h_*))\cdot\exp(-cn^{1/5}).$$
On the other hand, for every $\bar{h}\in\bar{\mathcal{H}}_{n,m^{\otimes2}}^{(1)}$, $\bar{h}$ is within an $O(n^{-1/4})$ distance from $\bar{h}_*$. As a result, due to the differentiability of $\Phi_2$ the Hessians $\mathbf{H}_{\bar{h}}$ and $\mathbf{H}_{\bar{h}_*}$ of the functions $\Phi_{\bar{h}}$ and $\Phi_{\bar{h}_*}$ also satisfy $\lVert\mathbf{H}_{\bar{h}}-\mathbf{H}_{\bar{h}_*}\rVert_{\text{OP}}=O(n^{-1/4})$. A direct computation of the Hessian $\mathbf{H}_{\bar{h}_*}$, verified by Mathematica yields:
\begin{equation}\label{det-Hessian}
\det(-\mathbf{H}_{\bar{h}_*})=\dfrac{2^{29}\cdot(d-1)^{16}\cdot(d^2-2d+2)}{(d-2)^8\cdot d^4}>0,
\end{equation}
which means that on a region of $\bar{h}_*$, $\mathbf{H}_{\bar{h}}$ is negative-definite and bounded away from a singular matrix. In other words, keeping in mind Lemma \ref{comp-fn-Taylor} as well, there exists a constant $c>0$ for which 
$$\Phi_{\bar{h}}(h)\le\Phi_2(\bar{h}^*)-c\lVert h-\bar{h}^*\rVert_2^2.$$
This, working in the same way as before, implies that if we set $$\mathcal{H}_{n,\bar{h}}^{(1)}=\left\{h\in B_{\bar{h}}:\lVert h-\bar{h}^*\rVert_2\le n^{-2/5}\right\},$$
then
$$\SL{h\in B_{\bar{h}}\setminus\mathcal{H}_{n,\bar{h}}^{(1)}}\E(Z_h)\le n^{\Theta(1)}\cdot\exp(n\Phi_2(\bar{h}^*))\cdot\exp(-cn^{1/5}).$$
If for some $h$, $\bar{h}\in\bar{\mathcal{H}}_{n,m^{\otimes2}}^{(1)}$ and $h\in\mathcal{H}_{n,\bar{h}}^{(1)}$, due to (\ref{h-contr-reg}) and Stirling's formula,
$$\E(Z_h)\sim\dfrac{2^{17}\cdot(d-1)^{10}}{(\pi n)^{9/2}\cdot d^9\cdot(d-2)^4}\cdot\exp(n\Phi_2(h)),$$
so
\begin{equation}\label{contr-good-h}
\SL{\bar{h}\in\bar{\mathcal{H}}_{n,m^{\otimes2}}^{(1)}}\SL{h\in\mathcal{H}_{n,\bar{h}}^{(1)}}\E(Z_h)\sim\dfrac{2^{17}\cdot(d-1)^{10}}{(\pi n)^{9/2}\cdot d^9\cdot(d-2)^4}\SL{\bar{h}\in\bar{\mathcal{H}}_{n,m^{\otimes2}}^{(1)}}\SL{h\in\mathcal{H}_{n,\bar{h}}^{(1)}}\exp(n\Phi_2(h))
\end{equation}
We claim that all the sums of the form
$$\exp(-n\Phi_2(\bar{h}^*))\SL{h\in\mathcal{H}_{n,\bar{h}}^{(1)}}\exp(n\Phi_2(h)),$$
for $\bar{h}\in\bar{\mathcal{H}}_{n,m^{\otimes2}}^{(1)}$, are within a $1+o(1)$ factor.\\\\ 
Indeed, for any such $\bar{h}$ and $h\in\mathcal{H}_{n,\bar{h}}^{(1)}$, $\Phi_2(h)=\Phi_2(\bar{h}^*)+\frac{1}{2}\langle h-\bar{h}^*,\mathbf{H}_{\bar{h}}(h-\bar{h}^*)\rangle+o(n^{-1})$, therefore
$$\exp(-n\Phi_2(\bar{h}^*))\SL{h\in\mathcal{H}_{n,\bar{h}}^{(1)}}\exp(n\Phi_2(h))\sim\SL{h\in\mathcal{H}_{n,\bar{h}}^{(1)}}\exp\left(\dfrac{n}{2}\langle h-\bar{h}^*,\mathbf{H}_{\bar{h}}(h-\bar{h}^*)\rangle\right).$$
Also, since $\lVert\mathbf{H}_{\bar{h}}-\mathbf{H}_{\bar{h}_*}\rVert_{\text{OP}}=O(n^{-1/4})$ as explained above, for any $h\in\mathcal{H}_{n,\bar{h}}^{(1)}$,
\begin{align*}
\langle h-\bar{h}^*,\mathbf{H}_{\bar{h}}(h-\bar{h}^*)\rangle&=\langle h-\bar{h}^*,\mathbf{H}_{\bar{h}_*}(h-\bar{h}^*)\rangle+\langle h-\bar{h}^*,(\mathbf{H}_{\bar{h}}-\mathbf{H}_{\bar{h}_*})(h-\bar{h}^*)\rangle\\&=\langle h-\bar{h}^*,\mathbf{H}_{\bar{h}_*}(h-\bar{h}^*)\rangle+o(n^{-1}),
\end{align*}
because $|\langle h-\bar{h}^*,(\mathbf{H}_{\bar{h}}-\mathbf{H}_{\bar{h}_*})(h-\bar{h}^*)\rangle|\le\lVert h-\bar{h}^*\rVert_2^2\cdot\lVert\mathbf{H}_{\bar{h}}-\mathbf{H}_{\bar{h}_*}\rVert_{\text{OP}}=O(n^{-4/5-1/4})=o(n^{-1}).$
Putting all of this together implies
\begin{align*}
\exp(-n\Phi_2(\bar{h}^*))\SL{h\in\mathcal{H}_{n,\bar{h}}^{(1)}}\exp(n\Phi_2(h))&\sim\SL{h\in\mathcal{H}_{n,\bar{h}}^{(1)}}\exp\left(-\dfrac{n}{2}\langle h-\bar{h}^*,\mathbf{H}_{\bar{h}}(h-\bar{h}^*)\rangle\right)
\\&\sim\SL{h\in\mathcal{H}_{n,\bar{h}}^{(1)}}\exp\left(-\dfrac{n}{2}\langle h-\bar{h}^*,\mathbf{H}_{\bar{h}_*}(h-\bar{h}^*)\rangle\right)
\\&\ \sim\exp(-n\Phi_2(h_*))\SL{h\in\mathcal{H}_{n,\bar{h}_*}^{(1)}}\exp(n\Phi_2(h))\\&\ \sim\dfrac{d^8\cdot(d-2)^4\cdot(\pi n)^3}{2^{29/2}\cdot(d-1)^8\cdot\sqrt{d^2-2d+2}},
\end{align*}
where the last equality follows because of (\ref{det-Hessian}) and Laplace's method. We combine everything to prove (\ref{sec-mom-Zm-reg}). We work as in the calculation of the first moment.
\begin{align*}
\SL{\bar{h}\in\bar{\mathcal{H}}_{n,m^{\otimes2}}^{(1)}}\SL{h\in\mathcal{H}_{n,\bar{h}}^{(1)}}\exp(n\Phi_2(h))&=\SL{\bar{h}\in\bar{\mathcal{H}}_{n,m^{\otimes2}}^{(1)}}\exp(n\Phi_2(\bar{h}^*))\cdot\left(\exp(-n\Phi_2(\bar{h}^*))\SL{h\in\mathcal{H}_{n,\bar{h}}^{(1)}}\exp(n\Phi_2(h))\right)
\\&\sim\dfrac{d^8\cdot(d-2)^4\cdot(\pi n)^3}{2^{29/2}\cdot(d-1)^8\sqrt{d^2-2d+2}}\SL{\bar{h}\in\bar{\mathcal{H}}_{n,m^{\otimes2}}^{(1)}}\exp(n\Phi_2(\bar{h}^*))
\\& \sim\dfrac{d^8(d-2)^{7/2}(\pi n)^{7/2}}{2^{16}\cdot(d-1)^{17/2}}\cdot\exp\left(n\Phi_2(h_*)-\dfrac{(d-1)(d-2)m^4}{6d^2n^3}\right),
\end{align*}
where the last equality is due to Laplace's method. Together with (\ref{contr-good-h}) and (\ref{sec-mom-expan-reg}), this implies
$$\E(Z_m^2)\sim\dfrac{2\cdot(d-1)^{3/2}}{\pi nd\cdot\sqrt{d-2}}\cdot\exp(n\Phi_2(h_*))\cdot\exp\left(-\dfrac{(d-1)(d-2)}{6d^2}\cdot\dfrac{m^4}{n^3}\right),$$
and since
$$\Phi_2(h_*)=\log(4)+d\log\cosh(\beta),$$
we have successfully proven (\ref{sec-mom-Zm-reg}).
\end{proof}
\subsection{Erd\"os-R\'enyi Graphs}
\subsubsection{First moment calculation}
For any $\sigma$, $\tilde{Z}_\sigma$ is a product of independent random variables and its expectation only depends on $m(\sigma)$. Trivially, the number of $u$ such that $\sigma_u=1$ is $\frac{n+m}{2}$, while the number of $u$ such that $\sigma_u=-1$ is $\frac{n-m}{2}$. Moreover, the number of pairs $\{u,v\}$ such that $\sigma_u\sigma_v=1$ is $\binom{\frac{n+m}{2}}{2}+\binom{\frac{n-m}{2}}{2}=\frac{n^2+m^2}{4}-\frac{n}{2}$, whereas the number of $\{u,v\}$ such that $\sigma_u\sigma_v=-1$ is $\frac{n+m}{2}\cdot\frac{n-m}{2}=\frac{n^2-m^2}{4}$. Therefore, if $m(\sigma)=m$, then since $\frac{e^{\pm\beta}}{\cosh(\beta)}-1=\pm\frac{1}{d}$,
\begin{align*}
\E(\tilde{Z}_\sigma)&=2^{-n}\left(1+\dfrac{1}{n}\right)^{\frac{n^2+m^2}{4}-\frac{n}{2}}\cdot\left(1-\dfrac{1}{n}\right)^{\frac{n^2-m^2}{4}}\\&\sim2^{-n}\cdot\exp\left(\dfrac{m^2}{2n}-\frac{3}{4}\right),
\end{align*}
as $\log(1\pm\frac{1}{n})=\pm\frac{1}{n}-\frac{1}{2n^2}+O(n^{-3})$.\\\\
Taylor expanding the function $H(x)=-\frac{1+x}{2}\log(1+x)-\frac{1-x}{2}\log(1-x)$ around $x=0$ implies that
\begin{align*}
\E(\tilde{Z}_m)=\dbinom{n}{\frac{n+m}{2}}\cdot\E(\tilde{Z}_\sigma)&\sim\dfrac{\sqrt{2}}{\sqrt{\pi n}}\cdot\exp\left(H\left(\frac{m}{n}\right)+\frac{m^2}{2n}-\frac34\right)\\&\sim\dfrac{\sqrt{2}}{\sqrt{\pi n}}\cdot\exp\left(-\dfrac{m^4}{12n^3}-\dfrac{3}{4}\right).
\end{align*}
Therefore,
$$\SL{|m|\le n^{3/4}}\E(\tilde{Z}_m)\gtrsim n^{1/4}$$
and for any $C>0$,
$$\SL{|m|\le Cn^{3/4}}\E(\tilde{Z}_m)\lesssim n^{-1/2}\SL{k=0}^C\SL{k\le|m|\cdot n^{-3/4}\le k+1}\exp(-\tfrac{m^4}{12n^3})\lesssim n^{1/4}\SL{k=0}^C\exp(-\tfrac{k^4}{12})\lesssim n^{1/4}.$$
It follows that indeed, $\tilde{c}_0>0$.\\\\ 
Also, because of Lemma \ref{comp-fn-Taylor}, for $|m|\ge Cn^{3/4}$, $\E(\tilde{Z}_m)\lesssim n^{-1/2}\cdot\exp(-c\frac{m^4}{n^3})$. This implies
$$\SL{|m|>Cn^{3/4}}\E(\tilde{Z}_m)\lesssim\SL{k\ge C}\SL{k\le|m|\cdot n^{-3/4}\le k+1}n^{-1/2}\exp(-c\tfrac{m^4}{n^3})\lesssim n^{1/4}\cdot\SL{k\ge C}\exp(-ck^4).$$
Choosing $C>0$ to be large enough finishes the proof of Lemma \ref{right-scal-ER}. The first statement of Proposition \ref{Zm-ER-mom} is already proven.
\subsubsection{Second moment calculation}
\begin{proof}[Proof of the second statement of Proposition \ref{Zm-ER-mom}]
For any $\sigma,\sigma'\in\{-1,1\}^n$ with $m(\sigma)=m(\sigma')=m$, set
$$t(\sigma,\sigma')=\dfrac{1}{\sqrt{n}}\cdot\left(\SL{u}\sigma_u\sigma_u'-\dfrac{m^2}{n}\right).$$
From now on, for simplicity purposes, if $\sigma,\sigma'$ are fixed, we will just use $t$ instead of $t(\sigma,\sigma')$. By the definition of $t$, the number of vertices $u$ such that: 
\begin{itemize}
\item
$(\sigma_u,\sigma_u')=(1,1)$ is $n\tilde{h}(1,1)=\frac{n+\sqrt{n}t+\frac{m^2}{n}+2m}{4}$.
\item 
$(\sigma_u,\sigma_u')=(-1,-1)$ is $n\tilde{h}(-1,-1)=\frac{n+\sqrt{n}t+\frac{m^2}{n}-2m}{4}$.
\item 
$(\sigma_u,\sigma_u')=(1,-1)$ is $n\tilde{h}(1,-1)=\frac{n-\sqrt{n}t-\frac{m^2}{n}}{4}$.
\item 
$(\sigma_u,\sigma_u')=(-1,1)$ is $n\tilde{h}(-1,1)=\frac{n-\sqrt{n}t-\frac{m^2}{n}}{4}$.
\end{itemize}
This means that the number of pairs $\{u,v\}$ such that:
\begin{itemize}
\item
$\sigma_u\sigma_v+\sigma'_u\sigma'_v=2$ is
\begin{align*}
\dbinom{\frac{n+\sqrt{n}t+\frac{m^2}{n}+2m}{4}}{2}+\dbinom{\frac{n+\sqrt{n}t+\frac{m^2}{n}-2m}{4}}{2}+2\cdot\dbinom{\frac{n-\sqrt{n}t-\frac{m^2}{n}}{4}}{2}\\=\ \dfrac{n^2+(\sqrt{n}t+\frac{m^2}{n})^2}{8}+\frac{m^2}{4}-\frac{n}{2}.
\end{align*}
\item 
$\sigma_u\sigma_v+\sigma_u'\sigma_v'=0$ is
\begin{align*}
\left(\frac{n+\sqrt{n}t+\frac{m^2}{n}+2m}{4}+\frac{n+\sqrt{n}t+\frac{m^2}{n}-2m}{4}\right)\cdot\left(\frac{n-\sqrt{n}t-\frac{m^2}{n}}{4}+\frac{n-\sqrt{n}t-\frac{m^2}{n}}{4}\right)
\\=\dfrac{n^2-(\sqrt{n}t+\frac{m^2}{n})^2}{4}
\end{align*}
\item 
$\sigma_u\sigma_v+\sigma_u'\sigma_v'=-2$ is
\begin{align*}
\frac{n+\sqrt{n}t+\frac{m^2}{n}+2m}{4}\cdot\frac{n+\sqrt{n}t+\frac{m^2}{n}-2m}{4}+\left(\frac{n-\sqrt{n}t-\frac{m^2}{n}}{4}\right)^2\\=\ \dfrac{n^2+(\sqrt{n}t+\frac{m^2}{n})^2}{8}-\dfrac{m^2}{4}.
\end{align*}
\end{itemize}
Keeping all of this in mind and combining it with the facts that $\frac{e^{\pm2\beta}}{\cosh(\beta)^2}-1=\pm\dfrac{2}{d}+\dfrac{1}{d^2}$ and $\frac{1}{\cosh(\beta)^2}-1=-\dfrac{1}{d^2}$ implies that
\begin{align*}
\E(\tilde{Z}_\sigma\tilde{Z}_{\sigma'})\sim&\ 4^{-n}\left(1+\tfrac{2d+1}{dn}\right)^{\frac{n^2+(\sqrt{n}t+\frac{m^2}{n})^2}{8}+\frac{m^2}{4}-\frac{n}{2}}\left(1-\tfrac{1}{dn}\right)^{\frac{n^2-(\sqrt{n}t+\frac{m^2}{n})^2}{4}}\left(1-\tfrac{2d-1}{dn}\right)^{\frac{n^2+(\sqrt{n}t+\frac{m^2}{n})^2}{8}-\frac{m^2}{4}}\\\sim&4^{-n}\cdot\exp\left[\left(\tfrac{2d+1}{dn}-\tfrac{(2d+1)^2}{2d^2n^2}\right)\left(\tfrac{n^2+(\sqrt{n}t+\frac{m^2}{n})^2}{8}+\tfrac{m^2}{4}-\tfrac{n}{2}\right)\right.\\&\left.+\left(-\tfrac{1}{dn}-\tfrac{1}{2d^2n^2}\right)\cdot\tfrac{n^2-(\sqrt{n}t+\tfrac{m^2}{n})^2}{4}+\left(-\tfrac{2d-1}{dn}-\tfrac{(2d-1)^2}{2d^2n^2}\right)\left(\tfrac{n^2+(\sqrt{n}t+\frac{m^2}{n})^2}{8}-\tfrac{m^2}{4}\right)\right]
\\ \sim& 4^{-n}\cdot\exp\left(\tfrac{m^2}{n}+\tfrac{t^2}{2d}+\tfrac{m^2t}{dn^{3/2}}+\tfrac{m^4}{2dn^3}-\tfrac{1+2d+6d^2}{4d^2}\right).
\end{align*}
For any $t$, let
$$\tilde{Z}_{m,t}=\SL{\substack{(\sigma,\sigma'):\ m(\sigma)=m(\sigma')=m\\t(\sigma,\sigma')=t}}\tilde{Z}_\sigma\tilde{Z}_{\sigma'}.$$
For any $t$,
$$\dfrac{n!}{\left(n\tilde{h}(1,1)\right)!\cdot\left(n\tilde{h}(-1,-1)\right)!\cdot\left(n\tilde{h}(1,-1)\right)!^2}\le n^{\Theta(1)}\cdot\exp(nH(\tilde{h})),$$
where $H(\tilde{h})$ is the Shannon entropy of the probability distribution $\tilde{h}$. So,
\begin{equation}\label{contr-t-far-1}
\E(\tilde{Z}_{m,t})\le n^{\Theta(1)}\cdot\exp\left[n\left(\log(4)+\tfrac{m^2}{n^2}+\tfrac{t^2}{2dn}+\tfrac{m^2t}{dn^{5/2}}+H(\tilde{h})\right)\right]
=n^{\Theta(1)}\cdot\exp\left(\tilde{\Phi}(\tfrac{m}{n},\tfrac{t}{\sqrt{n}})\right),
\end{equation}
where $\tilde{\Phi}$ is the function
\begin{align*}
\tilde{\Phi}(x,y)=&\ x^2+\frac{y^2}{2d}+\dfrac{x^2y}{d}-\dfrac{(1+x)^2+y}{4}\cdot\log\left((1+x)^2+y\right)\\&-\dfrac{(1-x)^2+y}{4}\cdot\log\left((1-x)^2+y\right)-\dfrac{1-x^2-y}{2}\cdot\log\left(1-x^2-y\right).
\end{align*}
Around (0,0), the Taylor expansion of $\tilde{\Phi}$ is
$$\tilde{\Phi}(x,y)=-\dfrac{d-1}{2d}\cdot y^2-\left(\dfrac{1}{6}-\dfrac{1}{2d}\right)\cdot x^4+\dfrac{1}{d}\cdot x^2y+O\left(y^3+x^6\right),$$
so there exists some constant $c>0$ for which
$$\tilde{\Phi}(x,y)\le-cy^2+O(y^3+x^4).$$
Viewing $\tilde{\Phi}$ as a function of $y$, with fixed $x$ Lemma \ref{comp-fn-Taylor} implies that $n\tilde{\Phi}(\tfrac{m}{n},\tfrac{t}{\sqrt{n}})\le-ct^2+O(\frac{m^4}{n^3})$.
Therefore, for $|t|>n^{1/10}$, by plugging this back into (\ref{contr-t-far-1}) we find that
\begin{equation}\label{contr-t-far-2}
\E(\tilde{Z}_{m,t})\le n^{\Theta(1)}\cdot\exp\left(-cn^{1/5}\right).
\end{equation}
On the other hand, for any $|t|\le n^{1/10}$, the possible choices for $\sigma$ and $\sigma'$ are
\begin{align*}
&\dfrac{n!}{\left(\frac{n+\sqrt{n}t+\frac{m^2}{n}+2m}{4}\right)!\cdot\left(\frac{n+\sqrt{n}t+\frac{m^2}{n}-2m}{4}\right)!\cdot\left(\frac{n-\sqrt{n}t-\frac{m^2}{n}}{4}\right)!^2}
\sim4^n\dfrac{4\sqrt{2}}{(\pi n)^{3/2}}\exp\left(-\frac{m^2}{n}-\frac{t^2}{2}-\frac{m^4}{6n^3}\right).
\end{align*}
So, keeping (\ref{contr-t-far-2}) in mind,
\begin{align*}
\E(\tilde{Z}_m^2)&\sim\dfrac{4\sqrt{2}}{(\pi n)^{3/2}}\cdot\exp\left(
\frac{m^4}{2dn^3}-\frac{m^4}{6n^3}-\frac{1+2d+6d^2}{4d^2}\right)\cdot\SL{|t|\le n^{1/10}}\exp\left(
-\frac{d-1}{2d}t^2+\frac{m^2}{dn^{3/2}}t\right)
\\&\sim\dfrac{2}{\pi n}\cdot\exp\left(
-\dfrac{m^4}{6n^3}+\dfrac{m^4}{2(d-1)n^3}-\dfrac{1+2d+6d^2}{4d^2}\right)\cdot\sqrt{\dfrac{d}{d-1}}.
\end{align*}
This proves the second statement of Proposition \ref{Zm-ER-mom}.
\end{proof}
\section{Cycle and path counts}\label{5}
\subsection{Cycle counts in $G_{n,d}$}
The goal of this subsection is to prove Proposition \ref{cyc-reg}. In fact, since statement 1 is very well-known, we only prove statement 2. The analog of this statement was proven in \cite{Fabian-Loick} for the anti-ferromagnetic Ising model, however there are some minor differences with our case, which is why we prove it here.\\\\
For any $h\in\mathcal{H}_{n,m}^{(1)}$ and $i\ge1$, we prove that
\begin{equation}\label{cyc-planted}
\dfrac{\E(Z_h\cdot Y_{i,n})}{\E(Z_h)}\sim1+(d-1)^{-i}.
\end{equation}
For any $C$ of length $i$ that has vertices $v_1,\dots,v_i$, we first plant $\sigma_C\in\{-1,1\}^C$, which also induces the numbers $k_{++},k_{+-}$ and $k_{--}$ of edges in $C$ that connect two vertices with $+1$'s, one with $+1$ and one with $-1$ and two with $-1$'s, respectively. We now construct the rest of the graph, in exactly the same way as in (\ref{A-form-reg}). There are $\binom{n-i}{\frac{n-i+m-m(\sigma_C)}{2}}$ ways to put the remaining spins on the vertices, $\binom{d\frac{n+m}{2}-2k_{++}-k_{+-}}{dnh(1,1)-2k_{++}}\cdot\binom{d\frac{n-m}{2}-2k_{--}-k_{+-}}{dnh(-1,-1)-2k_{--}}$ ways to put the spins on the half-edges and $(dnh(1,1)-2k_{++}-1)!!\cdot(dnh(-1,-1)-2k_{--}-1)!!\cdot(dnh(1,-1)-k_{+-})!$ ways to connect the half edges in an admissible way. Therefore, for any $C$,
\begin{align*}
\E(Z_h|C\in G)=\tfrac{1}{(dn-2i-1)!!}\SL{\sigma_C}&\tbinom{n-i}{\frac{n-i+m-m(\sigma_C)}{2}}\tbinom{d\frac{n+m}{2}-2k_{++}-k_{+-}}{dnh(1,1)-2k_{++}}\cdot\tbinom{d\frac{n-m}{2}-2k_{--}-k_{+-}}{dnh(-1,-1)-2k_{--}}\\&\cdot(dnh(1,1)-2k_{++}-1)!!\cdot(dnh(-1,-1)-2k_{--}-1)!!\\&\cdot(dnh(1,-1)-k_{+-})!
\end{align*}
Using (\ref{A-form-reg}), Stirling's formula and the fact that for $h\in\mathcal{H}_{n,m}^{(1)}$, $h-h_*=o(1)$, we find that
\begin{align*}
\dfrac{\E(Z_h|C\in G)}{\E(Z_h)}&\sim\left(\dfrac{dn}{2}\right)^i\SL{\sigma_C}4^i(dn)^{-i}h(1,1)^{k_{++}}h(-1,-1)^{k_{--}}h(1,-1)^{k_{+-}}
\\&\sim\SL{\sigma_C}(2h(1,1))^{k_{++}}(2h(-1,-1))^{k_{--}}(2h(1,-1))^{k+-}
\\&\sim\SL{\sigma_C}\dfrac{1}{(2\cosh(\beta))^i}\exp\left(\beta\SL{j=1}^i\sigma_j\sigma_{j+1}\right)
\\&\sim1+(d-1)^{-i},
\end{align*}
which proves (\ref{cyc-planted}). As we already explained in (\ref{contr-h-away}),
\begin{align*}
\E\left(Y_{i,n}\SL{h\in\mathcal{H}_{n,m}\setminus\mathcal{H}_{n,m}^{(1)}}Z_h\right)\le n^{\Theta(1)}\cdot\SL{h\in\mathcal{H}_{n,m}\setminus\mathcal{H}_{n,m}^{(1)}}\E(Z_h)&\le \exp(-cn^{1/5})\cdot n^{\Theta(1)}\cdot\E(Z_m)\\&=o(\E(Z_m)).
\end{align*}
Combining this with (\ref{cyc-planted}), we get that
\begin{align*}
\dfrac{\E(Z_m\cdot Y_{i,n})}{\E(Z_m)}=\dfrac{\SL{C\in G}\SL{h\in\mathcal{H}_{n,m}}\E(Z_h\mathbf{1}_{C\in G})}{\SL{h\in\mathcal{H}_{n,m}}\E(Z_h)}&\sim\dfrac{\frac{(d-1)^i}{2i}\SL{h\in\mathcal{H}_{n,m}^{(1)}}\E(Z_h|C\in G)}{\SL{h\in\mathcal{H}_{n,m}^{(1)}}\E(Z_h)}\\&\sim\frac{(d-1)^i+1}{2i}.
\end{align*}
Therefore, the $Y_{i,n}$ have the correct first moments under the planted measure. Proving the asymptotic independence, as well as the distribution of the limit relies on calculating the joint factorial moments of $(Y_{i,n})$ under the planted measure. As in the simple case we have demonstrated above, this comes down to showing that several sequences of vertices form cycles concurrently. Under the planted measure, the existence of short cycles with non-empty overlap is of order $O(1/n)$. Therefore, the joint factorial moment calculation comes down to a sum over distinct vertices, and this is a trivial generalization of the calculation we demonstated above.
\subsection{Joint path and cycle counts in $G(n,d/n)$}
The goal of this subsection is to prove Proposition \ref{cyc-path-ER}. The first step towards that will be to calculate the expectation and the variance of $X_{\ell,n}$, showing that the normalization performed in (\ref{hat-X-def}) makes sense. 
\begin{lem}\label{exp-var-path-cyc}
\begin{enumerate}
\item
The random variable $X_{\ell,n}$ has
\begin{align*}
&\E(X_{\ell,n})=(1+O(n^{-1}))\cdot\dfrac{1}{2}nd^\ell\ \ \text{and}\\ &\text{Var}(X_{\ell,n})=(1+O(n^{-1}))\cdot\dfrac{1}{2}nd^{2\ell}\dfrac{\ell^2}{d-1}\left(1+\gamma_\ell^{(2)}\right),
\end{align*}
where $\lim\limits_{\ell\to\infty}\gamma_\ell^{(2)}=0$.
\item 
Let $m$ be as in statement 4 of Proposition \ref{cyc-path-ER}. Then, under the planted measure $\PR_m^*$,
\begin{align*}
&\E_m^*(X_{\ell,n})=(1+O(n^{-1}))\cdot\dfrac{1}{2}nd^\ell\left(1+m^2n^{-2}\dfrac{\ell}{d-1}\left(1+\gamma_\ell^{(1)}\right)\right)\ \ \text{and}\\ &\text{Var}_m^*(X_{\ell,n})\sim\text{Var}(X_{\ell,n}),
\end{align*}
where $\lim\limits_{\ell\to\infty}\gamma_\ell^{(1)}=0$.
\end{enumerate}
\end{lem}
\begin{proof}
For any $\ell\in\N$, let $\vec{\mathcal{P}}_\ell=\{(v_0,v_1,\dots,v_\ell):v_i\in[n], v_i\neq v_j\ \forall i,j\}$ be the set of directed paths of length $\ell$ in the complete graph $K_n$ and $\mathcal{P}_\ell=\vec{\mathcal{P}_\ell}/\sim$, where for two paths $p^{(1)}\neq p^{(2)}\in\vec{\mathcal{P}}_\ell$ we write $p^{(1)}\sim p^{(2)}$ if $v_i^{(1)}=v_{\ell-i}^{(2)}$ for any $0\le i\le \ell$. $\mathcal{P}_\ell$ is the set of paths of length $\ell$ in the complete graph $K_n$. Then,
$$\E(X_{\ell,n})=\SL{P\in\mathcal{P}_\ell}\PR(P\in G)=|\mathcal{P}_\ell|\cdot\left(\dfrac{d}{n}\right)^\ell=(1+O(n^{-1}))\cdot\dfrac{1}{2}nd^\ell,$$
as $|\mathcal{P}_\ell|=\frac{|\vec{\mathcal{P}}_\ell|}{2}=\frac{1}{2}n(n-1)\cdots(n-\ell)=(\frac{1}{2}+O(n^{-1}))\cdot n^{\ell+1}$. Also,
\begin{align*}
\text{Var}(X_{\ell,n})&=\SL{\substack{(P,Q)\in\mathcal{P}_\ell^2\\ E(P\cap Q)\neq\varnothing}}\left[\PR(P,Q\in G)-\PR(P\in G)\cdot\PR(Q\in G)\right]\\&=(1+O(n^{-1}))\cdot\SL{k=1}^\ell\SL{\substack{(P,Q)\in\mathcal{P}_\ell^2\\|E(P\cap Q)|=k}}\left(\dfrac{d}{n}\right)^{2\ell-k}
\\&=(1+O(n^{-1}))\cdot\SL{k=1}^\ell|\{(P,Q)\in\mathcal{P}_\ell:|E(P\cap Q)|=k\}|\cdot\left(\dfrac{d}{n}\right)^{2\ell-k}.
\end{align*}
We now prove that $|\{(P,Q)\in\mathcal{P}_\ell^2:|E(P\cap Q)|=k\}|=(1+O(n^{-1}))\cdot\frac{1}{2}(\ell-k+1)^2\cdot n^{2\ell-k+1}$.\\\\ First of all, observe that the number of pairs of paths meeting in at least two different segments is $O(n^{2\ell-k})$. Indeed, since at least $k+2$ vertices are common, there are at most $2\ell+2-(k+2)=2\ell-k$ choices of new vertices to be made. So, most pairs of paths in the set $\{(P,Q)\in\mathcal{P}_\ell^2:|E(P\cap Q)|=k\}$ intersect in a single path of length $k$. Note that if we ask for $|\{(\vec{P},\vec{Q})\in\vec{\mathcal{P}}_\ell^2:|E(\vec{P}\cap \vec{Q})|=k\}|$ instead, this cardinality is $(1+O(n^{-1}))\cdot(\ell-k+1)^2\cdot n^{2\ell-k+1}$. Indeed, for every choice of $\vec{P}$, there are $\ell-k+1$ ways to choose the segment of $\vec{P}$ that will be common with $\vec{Q}$, another $\ell-k+1$ ways to choose which part of $\vec{Q}$ will cover the chosen part and another $(1+O(n^{-1})\cdot n^{\ell-k}$ ways to complete $\vec{Q}$. Every path $\vec{P}\in\vec{\mathcal{P}}_\ell$ is a choice of $P\in\mathcal{P}_\ell$ and a choice of a direction. For any $(P,Q)$ that intersect in a path of length $k$, there are two ways of choosing the direction of $P$ and then \textit{exactly one} way of choosing the direction of $Q$, as their common part already has a direction. Therefore, $|\{(P,Q)\in\mathcal{P}_\ell^2:|E(P\cap Q)|=k\}|=(1+O(n^{-1}))\cdot\frac{1}{2}(\ell-k+1)^2\cdot n^{2\ell-k+1}$, as we initially claimed. Thus:
\begin{align*}
\text{Var}(X_{\ell,n})&=\left(\frac12+O(n^{-1})\right)\cdot nd^{2\ell}\SL{k=1}^\ell(\ell-k+1)^2\cdot d^{-k}
\\&=(1+O(n^{-1}))\cdot\dfrac{1}{2}nd^{2\ell}\dfrac{\ell^2}{d-1}\left(1+\gamma_\ell^{(2)}\right),
\end{align*}
which proves the desired result.\\\\
We move on to the planted measure case. For any $P=(w_0,w_1,\dots,w_\ell)\in\mathcal{P}_\ell$ and $\sigma$ such that $m(\sigma)=m$,
$$\dfrac{\E(\tilde{Z}_\sigma|\ P\in G)}{\E(\tilde{Z}_\sigma)}=\dfrac{\frac{\exp\left(\beta\SL{i=0}^{\ell-1}\sigma_{w_i}\sigma_{w_{i+1}}\right)}{\cosh(\beta)^\ell}}{\PL{i=0}^{\ell-1}\left(1+\frac{\sigma_{w_i}\sigma_{w_{i+1}}}{n}\right)}=(1+O(n^{-1}))\cdot\frac{\exp\left(\beta\SL{i=0}^{\ell-1}\sigma_{w_i}\sigma_{w_{i+1}}\right)}{\cosh(\beta)^\ell}.$$
We calculate $\E_m^*(X_{\ell,n})$.
\begin{align*}
\E_m^*(X_{\ell,n})&=\dfrac{\E(\tilde{Z}_mX_{\ell,n})}{\E(\tilde{Z}_m)}=\dfrac{1}{\E(\tilde{Z}_m)}\cdot\SL{P\in\mathcal{P}_\ell}\E(\tilde{Z}_m\mathbf{1}_{P\in G})\\&=\dfrac{1}{\E(\tilde{Z}_m)}\cdot\SL{P\in\mathcal{P}_\ell}\PR(P\in G)\SL{\sigma:m(\sigma)=m}\E(\tilde{Z}_\sigma|\ P\in G)
\\&=(1+O(n^{-1}))\dfrac12nd^\ell\dbinom{n}{\frac{n+m}{2}}^{-1}\SL{\sigma:m(\sigma)=m}\dfrac{\exp\left(\beta\SL{i=0}^{\ell-1}\sigma_{w_i}\sigma_{w_{i+1}}\right)}{\cosh(\beta)^\ell}
\\&=(1+O(n^{-1}))\frac12nd^\ell\dbinom{n}{\frac{n+m}{2}}^{-1}\SL{\sigma_P\in\{-1,1\}^{V(P)}}\dbinom{n-\ell-1}{\frac{n-\ell-1+m-m(\sigma_P)}{2}}
\dfrac{\exp\left(\beta\SL{i=0}^{\ell-1}\sigma_{w_i}\sigma_{w_{i+1}}\right)}{\cosh(\beta)^\ell}
\\&=(1+O(n^{-1}))\dfrac12nd^\ell\SL{\sigma_P\in\{-1,1\}^{V(P)}}\dfrac{\exp\left(\beta\SL{i=0}^{\ell-1}\sigma_{w_i}\sigma_{w_{i+1}}\right)}{\cosh(\beta)^\ell}\cdot\PL{i=0}^\ell\dfrac{1+\sigma_{w_i}\frac{m}{n}}{2}
\\&=(1+O(n^{-1}))\frac12nd^\ell\cdot\E_{\mu_\ell}\left[\PL{i=0}^\ell\left(1+\sigma_{w_i}\frac{m}{n}\right)\right],
\end{align*}
where $\mu_\ell$ is the Ising model on the segment of length $\ell$, at inverse temperature $\beta=\tanh^{-1}(d^{-1})$. Therefore,
\begin{align*}
\E_m^*(X_{\ell,n})&=(1+O(n^{-1}))\frac12nd^\ell\cdot\E_{\mu_\ell}\left[\PL{i=0}^\ell\left(1+\sigma_{w_i}\frac{m}{n}\right)\right]
\\&=(1+O(n^{-1}))\dfrac{1}{2}nd^\ell\cdot\left(1+\dfrac{m^2}{n^2}\SL{0\le i_1<i_2\le\ell}\E_{\mu_\ell}(\sigma_{w_{i_1}}\sigma_{w_{i_2}})\right)
\\
&=(1+O(n^{-1}))\frac12 nd^{\ell}\bigg(1+\frac{m^2}{n^2}\sum_{0\le i_1<i_2\le\ell}d^{-(i_2-i_1)}+O(n^{-1})\bigg)\\
&=(1+O(n^{-1}))\frac12 nd^{\ell}\bigg(1+\frac{m^2}{n^2}\sum_{k=1}^{\ell}(\ell+1-k)d^{-k}\bigg)\\
&=(1+O(n^{-1}))\dfrac{1}{2}nd^\ell\left(1+m^2n^{-2}\dfrac{\ell}{d-1}\left(1+\gamma_\ell^{(1)}\right)\right).
\end{align*}
Finally, we show that the variance under the planted measure differs from the variance under the regular measure only by a $1+o(1)$ factor. Just as in the previous cases,
\begin{align*}
\text{Var}^*_m(X_{\ell,n})&\sim\SL{\substack{(P,Q)\in\mathcal{P}_\ell^2\\E(P\cap Q)\neq\varnothing}}\PR_m^*(P,Q\in G)
\\&\sim\dfrac{1}{\E(\tilde{Z}_m)}\cdot\SL{\substack{(P,Q)\in\mathcal{P}_\ell^2\\E(P\cap Q)\neq\varnothing}}\PR(P,Q\in G)\SL{\sigma:m(\sigma)=m}\E(\tilde{Z}_\sigma|\ P,Q\in G).
\end{align*}
However, for $P,Q$ whose overlap is a segment, and therefore their union is a tree,
\begin{align*}
\dfrac{1}{\E(\tilde{Z}_m)}\cdot\SL{\sigma:m(\sigma)=m}\E(\tilde{Z}_\sigma|\ P,Q\in G)&\sim\ \dbinom{n}{\frac{n+m}{2}}^{-1}\SL{\sigma:m(\sigma)=m}\dfrac{\exp\left(\beta\SL{\{w,w'\}\in E(P\cup Q)}\sigma_w\sigma_{w'}\right)}{\cosh(\beta)^{|E(P\cup Q)|}}
\\&\sim\SL{\sigma_{P\cup Q}\in\{-1,1\}^{V(P\cup Q)}}\dfrac{\exp\left(\beta\SL{\{w,w'\}\in E(P\cup Q)}\sigma_w\sigma_{w'}\right)}{2^{|V(P\cup Q)|}\cdot\cosh(\beta)^{|E(P\cup Q)|}}\sim1.
\end{align*}
Also, the contribution to the variance from pairs $(P,Q)$ whose union contains a cycle is $O(1)$, as we explained in the computation of $\text{Var}(X_{\ell,n})$. It follows that
$$\text{Var}^*_m(X_{\ell,n})\sim\SL{\substack{(P,Q)\in\mathcal{P}_\ell^2\\E(P\cap Q)\neq\varnothing}}\PR(P,Q\in G)\sim\text{Var}(X_{\ell,n}).$$
We have proven Lemma \ref{exp-var-path-cyc}.
\end{proof}
Next, we calculate the probability that a cycle of length $i$ is contained in $G$ under $\PR_m^*$.
\begin{lem}\label{prob-cyc-pl}
Suppose $m$ is as in Lemma \ref{Zm-ER-mom}. Then, if $C$ is a cycle of length $i\ge3$,
$$\PR_m^*(C\in G)\sim\left(\dfrac{d}{n}\right)^i\cdot(1+d^{-i}).$$
\end{lem}
\begin{proof}
We work in a similar fashion as before.
\begin{align*}
\PR_m^*(C\in G)&=\dbinom{n}{\frac{n+m}{2}}^{-1}\SL{\sigma:m(\sigma)=m}\dfrac{\E(\tilde{Z}_\sigma\mathbf{1}_{C\in G})}{\E(\tilde{Z}_\sigma)}
\\&\sim\PR(C\in G)\SL{\sigma_C\in\{-1,1\}^{V(C)}}\dfrac{\exp\left(\beta\SL{\{v,v'\}\in E(C)}\sigma_v\sigma_{v'}\right)}{(2\cosh(\beta))^i}
\\&\sim\left(\dfrac{d}{n}\right)^i\cdot(1+\tanh(\beta)^i)
\\&\sim\left(\dfrac{d}{n}\right)^i(1+d^{-i}).
\end{align*}
The Lemma follows.
\end{proof}
We now prove Proposition \ref{cyc-path-ER}. The method of proof is largely based on the proof of asymptotic normality of the path counts that is presented in \cite{Janson2000}, which, in turn, is based on the proof of Ruci\'nski in \cite{Rucinski1988}.
\begin{proof}[Proof of Proposition \ref{cyc-path-ER}]
First, it should be noted that due to Lemma \ref{exp-var-path-cyc},
\begin{align*}
\E(\widehat{X}_{\ell,n})\xrightarrow[n\to\infty]{}0\ \ \text{and}\ \ \text{Var}(\widehat{X}_{\ell,n})\xrightarrow[n\to\infty]{}1+\gamma_\ell^{(2)}
\end{align*}
and if $m\cdot n^{-3/4}\xrightarrow[n\to\infty]{}x$, 
$$
\E_m^*(\widehat{X}_{\ell,n})\xrightarrow[n\to\infty]{}\dfrac{x^2}{\sqrt{2(d-1)}}\left(1+\gamma_\ell^{(1)}\right)\ \ \text{and}\ \ \text{Var}_m^*(\widehat{X}_{\ell,n})\xrightarrow[n\to\infty]{}1+\gamma_\ell^{(2)}.
$$
Therefore, the parameters of the limits are the ones claimed in Proposition \ref{cyc-path-ER}. We prove that for any $r_1,r_3,\dots,r_s\in\N$, as $n\to\infty$,
\begin{equation}\label{mom-ER}
\E\left[(X_{\ell,n}-\E(X_{\ell,n}))^{r_1}\cdot\PL{i=3}^s[Y_{i,n}]_{r_i}\right]\sim\begin{cases}
(r_1-1)!!\cdot\text{Var}(X_{\ell,n})^{r_1/2}\cdot\PL{i=3}^s\tilde{\lambda}_i^{r_i},\ \ &\text{if}\ r_1\ \text{is even}\\
o(1)\cdot\text{Var}(X_{\ell,n})^{r_1/2},&\text{if}\ r_1\ \text{is odd,}
\end{cases}
\end{equation}
where for $x\in\Z_{\ge0}, [x]_r=x(x-1)\cdots(x-r+1)$. The Poisson and normal distributions have exponential tails, therefore moment convergence is enough to prove convergence in distribution. So, proving (\ref{mom-ER}) will imply the result. In the case of the planted measure, this relation will become
$$\E_m^*\left[(X_{\ell,n}-\E_m^*(X_{\ell,n}))^{r_1}\cdot\PL{i=3}^s[Y_{i,n}]_{r_i}\right]\sim\begin{cases}
(r_1-1)!!\cdot\text{Var}_m^*(X_{\ell,n})^{r_1/2}\cdot\PL{i=3}^s\tilde{\mu}_i^{r_i},\ \ &\text{if}\ r_1\ \text{is even}\\
o(1)\cdot\text{Var}_m^*(X_{\ell,n})^{r_1/2},&\text{if}\ r_1\ \text{is odd.}
\end{cases}
$$
The proof for the planted measure is identical, so from now on we only work under the regular measure $\PR$ and only note the slight modifications if and when this is necessary. Notice that due to Lemma \ref{exp-var-path-cyc}, $\text{Var}(X_{\ell,n})=\Theta(n)$. Set $r=r_1+r_3+\cdots+r_s$ and observe that
\begin{equation}\label{path-cyc-exp}
\E\left[(X_{\ell,n}-\E(X_{\ell,n}))^{r_1}\cdot\PL{i=3}^s[Y_{i,n}]_{r_i}\right]
=\SL{G_1,G_2,\dots,G_r}\E\left[\PL{i=1}^{r-r_1}\mathbf{1}_{G_i}\cdot\PL{i=r-r_1+1}^{r}(\mathbf{1}_{G_i}-\E(\mathbf{1}_{G_i}))\right]
\end{equation}
where the sum is over all $r$-tuples $(G_1,\dots,G_r)$ of subgraphs of $K_n$ for which:
\begin{itemize}
\item 
For every $i\in\{3,\dots,s\}$, $G_{r_3+\cdots+r_{i-1}+1},\dots,G_{r_3+\cdots+r_i}$ are pairwise distinct cycles of length $i$ (for $i=3$, $r_3+\cdots+r_{i-1}=0$).
\item
$G_{r-r_1+1},\dots,G_r$ are paths of length $\ell$ (not necessarily distinct).
\end{itemize}
We introduce the graph $L=L(G_1,\dots,G_r)$, which has vertices $G_1,\dots,G_r$ and for $i\neq j$, $G_i\sim_LG_j$ if, and only if, $G_i$ and $G_j$ have a common vertex in $K_n$. By $L'$, we will denote the graph induced by the vertices $G_{r-r_1+1},\dots,G_r$. Also, for $1\le i\le r$, let $F_i$ be the (possibly empty) graph which has the same isomorphism class as $\left(\bigcup\limits_{j=1}^{i-1}G_i\right)\cap G_i$. Setting
$$T(G_1,\dots,G_r):=\E\left[\PL{i=1}^{r-r_1}\mathbf{1}_{G_i}\cdot\PL{i=r-r_1+1}^{r}(\mathbf{1}_{G_i}-\E(\mathbf{1}_{G_i}))\right],$$
we may write (\ref{path-cyc-exp}) as
$$\E\left[(X_{\ell,n}-\E(X_{\ell,n}))^{r_1}\cdot\PL{i=3}^s[Y_{i,n}]_{r_i}\right]
=\SL{L}\SL{\substack{G_1,\dots,G_r\\ \text{correct}\ L}}T(G_1,\dots,G_r).$$
Also, we will denote by $T_m^*(G_1,\dots,G_r)$ the expectation of the same random variable under the planted measure. Let $\mathcal{L}_0$ be the class of graphs $L$ in which $L'$ is a perfect matching and the rest of the vertices are isolated.
\begin{lem}\label{L-notin-L0}
If $L\notin\mathcal{L}_0$, then
\begin{equation}\label{L-neq-L0}
\SL{\substack{G_1,\dots,G_r\\ \text{correct}\ L}}T(G_1,\dots,G_r)=o\left(n^{r_1/2}\right).
\end{equation}
\end{lem}
\begin{proof}[Proof of Lemma \ref{L-notin-L0}]
We bound $T(G_1,\dots,G_r)$ as follows:
\begin{align}\label{T-bound}
\nonumber|T(G_1,\dots,G_r)|&\le\E\left[\PL{i=1}^{r-r_1}\mathbf{1}_{G_i}\cdot\PL{i=r-r_1+1}^{r}|\mathbf{1}_{G_i}-\E(\mathbf{1}_{G_i})|\right]
\\\nonumber&\le\E\left[\PL{i=1}^{r-r_1}\mathbf{1}_{G_i}\cdot\PL{i=r-r_1+1}^{r}(\mathbf{1}_{G_i}+\E(\mathbf{1}_{G_i}))\right]
\\&\le2^{r_1}\cdot\E\left[\PL{i=1}^{r-r_1}\mathbf{1}_{G_i}\cdot\PL{i=r-r_1+1}^{r}\mathbf{1}_{G_i}\right],
\end{align}
where the last inequality holds because of the FKG inequality.\\\\
If $L\notin\mathcal{L}_0$, then it falls into one of the following categories:
\begin{enumerate}
\item 
At least one of the $G_{r-r_1+1},\dots,G_r$ is an isolated vertex in $L$.
\item 
None of the vertices $G_{r-r_1+1},\dots,G_r$ is isolated and: $L'$ is not a perfect matching or at least one of them is connected to one of the vertices $G_1,\dots,G_{r-r_1}$.
\item 
$L'$ is a perfect matching, its vertices are isolated from the rest of $L$ and there is at least one edge between the vertices $G_1,\dots,G_{r-r_1}$.
\end{enumerate}
We show (\ref{L-neq-L0}) in each one of these cases. If $L$ falls into category 1, $T(G_1,\dots,G_r)=0$, because one of the factors is independent of all the others and has expectation 0. Assume that $L$ is in category 2 or 3, and the $(F_i)_i$ are specified. Let $t=\ell r_1+\SL{i=3}^sir_i$ be the total number of vertices that must be chosen in $K_n$. The number of choices of the $G_1,\dots,G_r$ is $O(n^{t-\sum_{i}v(F_i)})$, and the probability for $G_1,\dots,G_r$ to be present in each of these choices is $O(n^{-t+r_1+\sum_ie(F_i)})$. Therefore, because of (\ref{T-bound}), it suffices to show that in cases 2 and 3 above,
\begin{equation}\label{v-e-Fi}
r_1+\SL{i}(e(F_i)-v(F_i))<r_1/2\ \Leftrightarrow\SL{i}(v(F_i)-e(F_i))>r_1/2.
\end{equation}
First, observe that every $F_i$, as a subgraph of $G_i$, has $v(F_i)\ge e(F_i)$, and if $i>r-r_1$ and $F_i\neq\varnothing$, $F_i$ is a forest, so $v(F_i)\ge e(F_i)+1$.\\\\
If $L$ falls into category 3, observe that $$\SL{i=r-r_1+1}^r(v(F_i)-e(F_i))=r_1/2.$$
Let $j\in[r-r_1]$ be the smallest $j$ for which $F_j\neq\varnothing$. We claim that $F_j$ is a forest. Indeed, since the $(G_i)_{i=1}^{r-r_1}$ are pairwise distinct, $F_j=G_j$ can only happen if $G_j$ connects to two $G_{i_1},G_{i_2}$ with $i_1<i_2<j$ and $V(G_{i_1})\cap V(G_{i_2})\neq\varnothing$, contradicting the minimality of $j$. Since $F_j$ is a non-empty forest, $v(F_j)>e(F_j)$, concluding the proof of (\ref{v-e-Fi}).\\\\
Suppose $L$ falls into category 2. We will show that
$$\SL{i=r-r_1+1}^r(v(F_i)-e(F_i))>r_1/2.$$
Let $S$ be the set of $i\in[r-r_1+1,r]$ such that $F_i\neq\varnothing$, $c(L')$ be the number of connected components of $L'$ and $n(L')$ be the number of isolated points of $L'$. Any $i$ for which $G_i$ is an isolated point of $L'$ is in $S$, as the $G_i$ must have a neighbor in the $G_1,\dots,G_{r-r_1}$. Also, in the rest $c(L')-n(L')$ connected components that contain at least two vertices, at most one vertex per connected component is not in $S$. We observe that $2c(L')-n(L')\le r_1$, therefore
$$\SL{i=r-r_1+1}^r(v(F_i)-e(F_i))\ge|S|\ge n(L')+r_1-(c(L')-n(L'))\ge r_1/2+3n(L')/2.$$
If $n(L')\ge1$, then we have proven the desired result. Otherwise, $n(L')=0$ and $c(L')=r_1/2$, so $L'$ is a perfect matching. Due to the fact that $L$ is in category 2, there exists a $j\in S$ such that $G_j$ is connected to one of the $G_1,G_2,\dots,G_{r-r_1}$. Pick the smallest such $j$. Assume $G_j$ is connected to $G_i$ for some $1\le i \le r-r_1$ and to $G_{j'}$, for some $j'>r-r_1$. If $j>j'$, then $F_j$ is a forest with at least two connected components, one formed because of the intersection with $G_i$ and one because of the intersection with $G_{j'}$ (due to the minimality of $j'$, $G_j$ does not connect to any vertex in $(G_i)_{i=1}^{r-r_1}$. Therefore, $v(F_j)-e(F_j)\ge2$ and now (\ref{v-e-Fi}) is proven. On the other hand, if $j<j'$, then $G_j$ is the vertex appearing first in its connected component in $L'$ and is still in $S$. Therefore, $|S|\ge1+c(L')>r_1/2$ and (\ref{v-e-Fi}) follows in this case as well.\\\\
Since there are only $O(1)$ number of choices for $L$ and $(F_i)_i$, (\ref{L-neq-L0}) has been proven.
\end{proof}
\begin{lem}\label{L-in-L0}
For the class $\mathcal{L}_0$ we have $|\mathcal{L}_0|=(r_1-1)!!$ and for any $L_0\in\mathcal{L}_0$,
\begin{equation}\label{L-eq-L0}
\SL{\substack{G_1,\dots,G_r\\ L=L_0}}T(G_1,\dots,G_r)=(1+o(1))\cdot\text{Var}(X_{\ell,n})^{r_1/2}\cdot\PL{i=3}^s\tilde{\lambda}_i^{r_i}.
\end{equation}
Under the planted measure, the relation is the same except the $\tilde{\lambda_i}$ become $\tilde{\mu}_i$.
\end{lem}
\begin{proof}[Proof of Lemma \ref{L-in-L0}] The fact that $|\mathcal{L}_0|=(r_1-1)!!$ is easy to see, as there are exactly $(r_1-1)!!$ matchings of $r_1$ vertices. If $L=L_0$, then
\begin{equation}\label{T-eq}
T(G_1,\dots,G_r)=\PL{j=1}^{r-r_1}\PR(G_j\in G)\cdot\PL{j=r-r_1+1}^r\sqrt{\text{Cov}(\mathbf{1}_{G_j},\mathbf{1}_{G_{v(j)}})},
\end{equation}
where for any $j$, $v(j)$ is the unique vertex in $L_0$ for which $G_j\cap G_{v(j)}\neq\varnothing$. Fix $G_1,\dots,G_{r-r_1}$ which are pairwise disjoint and let $W=V(G_1\cup\dots\cup G_{r-r_1}), w=\SL{i=3}^sir_i$ be the set and the total number of vertices used by one of the $G_1,\dots,G_{r-r_1}$. We calculate: 
\begin{align}\label{Var-calc}
\nonumber\text{Var}(X_{\ell,n})^{r_1/2}\sim\text{Var}(X_{\ell,n-w})^{r_1/2}=&\SL{\substack{G_{r-r_1+1},\dots,G_r\\ G_{r-r_1+j}\cap W=\varnothing \\ G_{r-r_1+2j-1}\cap G_{r-r_1+2j}\neq\varnothing\ \forall j}}\PL{j=1}^{r_1/2}\text{Cov}(\mathbf{1}_{G_{r-r_1+2j-1}},\mathbf{1}_{G_{r-r_1+2j}})
\\\nonumber=&\SL{\substack{G_{r-r_1+1},\dots,G_r\\ G_j\cap W=\varnothing \\ G_j\cap G_{v(j)}\neq\varnothing\ \forall j}}\ \PL{j=r-r_1+1}^{r}\sqrt{\text{Cov}(\mathbf{1}_{G_{j}},\mathbf{1}_{G_{v(j)}})}
\\\nonumber=&\SL{\substack{G_{r-r_1+1},\dots,G_r \\L=L_0}}\ \PL{j=r-r_1+1}^r\sqrt{\text{Cov}(\mathbf{1}_{G_j},\mathbf{1}_{G_{v(j)}})}\\&+\SL{\substack{G_{r-r_1+1},\dots,G_r \\G_j\cap G_{v(j)}\neq\varnothing\ \forall j\\ L\neq L_0}}\ \PL{j=r-r_1+1}^r\sqrt{\text{Cov}(\mathbf{1}_{G_j},\mathbf{1}_{G_{v(j)}})}.
\end{align}
When $G_j\cap G_{v(j)}\neq\varnothing$ for all $j$ and $L\neq L_0$, we know that $L\notin\mathcal{L}_0$. This case was handled in Lemma \ref{L-notin-L0}, in the special case $r_3=\cdots=r_s=0$: At first, observe that
$$\left|\PL{j=r-r_1+1}^r\sqrt{\text{Cov}(\mathbf{1}_{G_j},\mathbf{1}_{G_{v(j)}})}\right|\le2^{r_1}\cdot\E\left[\PL{j=r-r_1+1}^r\mathbf{1}_{G_j}\right]$$
and then we proceed as we did once we proved (\ref{T-bound}) to prove that
$$\SL{\substack{G_{r-r_1+1},\dots,G_r \\G_j\cap G_{v(j)}\neq\varnothing\ \forall j\\ L\neq L_0}}\ \PL{j=r-r_1+1}^r\sqrt{\text{Cov}(\mathbf{1}_{G_j},\mathbf{1}_{G_{v(j)}})}=o(n^{r_1/2})=o(\text{Var}(X_{\ell,n})^{r_1/2}).$$
Combining this with (\ref{Var-calc}) it follows that for any $G_1,\dots,G_{r-r_1}$,
\begin{align*}
\SL{\substack{G_{r-r_1+1},\dots,G_r\\ L=L_0}}T(G_1,\dots,G_r)&\sim\text{Var}(X_{\ell,n})^{r_1/2}\cdot\PL{j=1}^{r-r_1}\PR(G_j\in G)\\&\sim\text{Var}(X_{\ell,n})^{r_1/2}\cdot\left(\dfrac{d}{n}\right)^w
\end{align*}
Because of Lemma \ref{prob-cyc-pl}, under the planted measure, we may write
$$\SL{\substack{G_{r-r_1+1},\dots,G_r\\ L=L_0}}T_m^*(G_1,\dots,G_r)\sim\text{Var}(X_{\ell,n})^{r_1/2}\left(\dfrac{d}{n}\right)^w\PL{i=3}^s(1+d^{-i})^{r_i}.$$
Summing over all possible choices of $G_1,\dots,G_{r-r_1}$ implies (\ref{L-eq-L0}), as for each cycle of length $i$, the number of ways to choose it in the complete graph is $\sim\frac{n^i}{2i}$.
\end{proof}
Using Lemmas \ref{L-in-L0} and \ref{L-notin-L0}, equation (\ref{mom-ER}), and therefore, Proposition \ref{cyc-path-ER}, follows.
\end{proof}

\end{document}